\documentclass{amsproc}
\usepackage{amsmath,amssymb}
\usepackage[a4paper, margin=1in]{geometry}
\usepackage{tikz-cd}
\usepackage{array}
\usepackage{multirow}
\usepackage{enumitem}

\newtheorem{theorem}{Theorem}
\newtheorem{lemma}{Lemma}
\newtheorem{definition}{Definition}
\newtheorem{corollary}{Corollary}
\newtheorem{proposition}[theorem]{Proposition}
\newtheorem{remark}{Remark}

\newtheorem{example}{Example}

\numberwithin{equation}{section}
\newtheorem*{proof*}{Proof}
\newcommand{\hf}{\hspace{0.5em}}

\begin{document}
	
	\title[Harish-Chandra Theorem for Okado-Yamane Quantum Groups] {Harish-Chandra Theorem for the Multi-Parameter \\ Quantum Groups of Okado-Yamane Type}

	\author{Kaixiang \textsc{Chen}}
	\address{School of Mathematical Sciences, \\
		East China Normal University, \\
		Shanghai 200241, China}
	\email{51265500019@stu.ecnu.edu.cn}
	
	\author{Naihong \textsc{Hu}}
	\address{School of Mathematical Sciences,
		\\ MOE Key Laboratory of Mathematics and Engineering
		\\ Applications \& Shanghai Key Laboratory of PMMP,
		\\ East China Normal University,
		\\ Shanghai 200241, China}
	\email{nhhu@math.ecnu.edu.cn}
	
	\author{Hengyi \textsc{Wang}}
	\address{School of Mathematical Sciences,
		\\ East China Normal University,
		\\ Shanghai 200241, China}
	\email{52265500001@stu.ecnu.edu.cn}
	
	

	\thanks{This work is supported by the NNSF of China (Grant No. 12171155), and in part by the Science and Technology Commission of Shanghai Municipality (Grant No. 22DZ2229014).}

	\date{2026.03.20}

	\begin{abstract}
		This paper is devoted to studying the centre of the multi-parameter quantum group $U_{q,G}(\mathfrak{g})$ introduced by Okado and Yamane, where $\mathfrak{g}$ is a complex simple Lie algebra, and all parameters lie in general position. We mainly establish the Harish-Chandra theorem, proving that the Harish-Chandra homomorphism is an isomorphism; in particular, we determine the centre $Z(U_{q,G})\cong (U^0_\flat)^W$ is isomorphic to a polynomial algebra or a quotient algebra of a polynomial algebra. The same result holds for the $(U^0_\flat)^W$ of the two-parameter quantum group $U_{r,s}(\mathfrak{g})$.
	\end{abstract}
	
	\maketitle

	\section{Introduction.}
	In mathematics and theoretical physics, ``quantum group'' $U_q(\mathfrak{g})$ refers to a class of Hopf algebras that are neither commutative nor cocommutative. One of the most well-known and widely studied classes of quantum groups is the Drinfeld-Jimbo presentation, which arises as a $q$-deformation of the universal enveloping algebra $U(\mathfrak{g})$ of a complex simple Lie algebra $\mathfrak{g}$. This construction was independently introduced by Drinfeld \cite{D86} and Jimbo \cite{J86} in the mid of 1980s, and has played a fundamental role in the development of quantum integrable systems, representation theory, and the theory of braid group actions etc. (see \cite{CP94, D86, J96, K95, L93}).
	In 1990, Okado and Yamane \cite{OY90} introduced a class of multi-parameter quantum groups $U_{q,G}(\mathfrak{g})$, defined via the parameter matrix $G=(q_{ij})$ and modified commutation relations, and investigated the basic $R$-matrix of $U_{q,G}(\mathfrak{\hat{\mathfrak{sl}}}_{n})$. Subsequently, Hayashi \cite{H92} showed that it admits a unique bilinear pairing and similar highest weight representations. In fact, $U_{q,G}(\mathfrak{g})$ can be realized as a specialization of the general multi-parameter quantum groups constructed in the second author's joint work with Pei and Rosso \cite{PHR10} under the appropriate parameter constraints.
	
	Several works on the centre of quantum groups have been developed over the last three decades. In 1990, Rosso \cite{R90} defined a significant $ad$-invariant bilinear form on $U_q(\mathfrak{g})$ with a generic $q$, where $\mathfrak{g}$ is a finite-dimensional simple Lie algebra. This form, often referred to as the  {Rosso form} or  {quantum Killing form}, has paved the way for the theorem of quantum Harish-Chandra isomorphism (later completed by Tanisaki \cite{T90, T92} and Joseph-Letzter \cite{JL92} etc. through different approaches).
	The centre $Z(U_q)$ is a polynomial algebra for the most of types and a quotient thereof in the remaining types, as studied by Li-Xia-Zhang \cite{LXZ16, LXZ18}.
	For quantum groups of weight lattice types $\breve{U}_q(\mathfrak{g})$, their centres are all isomorphic to polynomial algebras. (First proposed by Etingof \cite{E95}; a complete proof with explicit central element construction was later provided by Dai \cite{D23}, building upon the work of R. B. Zhang et al. \cite{GZB91, ZGB91}).
	In the broader context of quantum doubles of the bosonization of Nichols algebras of (finite) diagonal type, including multiparameter quantum groups, quantum superalgebras, and Lusztig's small quantum groups with free Cartan part, Batra and Yamane \cite{BY15, BY18, BY20} established the general Harish-Chandra-type theorem for (skew) centres and proposed a conjectural basis for these algebras.
	Their work provides a powerful, unified framework: \cite[Thm.~10.4]{BY18} characterizes the image of the Harish-Chandra homomorphism as an abstract subalgebra $\mathfrak{B}^{\chi,\pi}_\omega$ defined by a system of equations, while \cite[Conj.~3.13]{BY20} proposes that the centre should admit a natural basis parameterized by finite-dimensional weight modules. However, for concrete families of quantum groups, the explicit structure of the centre and the verification of this conjectural basis often require further investigation.
	
	Recently, Hu and Wang \cite{HW25} established the Harish-Chandra theorem for two-parameter quantum groups $U_{r,s}(\mathfrak{g})$ with detailed centre structure in the case when two-parameter $r, s$ are generic and central elements, and the open question in the odd-rank case solved by Xia \cite{X25} (with the modified relations in his definition of two-parameter quantum groups):
	\begin{proposition} \cite{HW25, X25} \label{HW25}
		Let $n= \operatorname{rank} (\mathfrak{g})$, parameters $r$ and $s$ be in general position, $U=U_{r,s}(\mathfrak{g})$, and $\breve{U}=\breve{U}_{r,s}(\mathfrak{g})$ be the weight lattice type of $U$, then we have the following Harish-Chandra theorem:
		\begin{enumerate}
			\item The Harish-Chandra homomorphism $\xi:\; Z(U) \rightarrow U^0$ is injective. The image $\xi(Z(U)) = (U^0_\flat)^W$ in the case when $n$ is even, and $\xi (Z(U)) \supseteq (U^0_\flat)^W \otimes \mathbb{K}[z_{*},z_{*}^{-1}]$ in the case when $n$ is odd.
			
			\item The Harish-Chandra homomorphism $\breve{\xi}:\; Z(\breve{U}) \rightarrow \breve{U}^0$ is injective. The image
			\begin{equation*}
				\xi(Z(\breve{U})) =
				\begin{cases}	
					(\breve{U}^0_\flat)^W \hspace{2.1cm} = \mathbb{K}[z_{\varpi_1},\cdots,z_{\varpi_n}], &\textnormal{if $n$ is even,}\\
					(\breve{U}^0_\flat)^W \otimes \mathbb{K}[z_*^{\frac{1}{\ell}}, z_*^{-\frac{1}{\ell}}] = \mathbb{K}[z_{\varpi_1},\cdots,z_{\varpi_n}] \otimes \mathbb{K}[z_*^{\frac{1}{\ell}}, z_*^{-\frac{1}{\ell}}],
					&\textnormal{if $n$ is odd,}
				\end{cases}
			\end{equation*}
			where $\varpi_i$ is the $i$-th fundamental weight (shown in Appendix A), central element $z_\lambda$ is obtained by the Rosso form realization of quantum trace on weight module $L(\lambda)$; $z_*$ is an additional invertible central generator (degenerating to the $1$ in the one-parameter specialization), and $\ell=2$, except $\ell=4$ for $D_{2k+1}$.
		\end{enumerate}
	\end{proposition}
	
	Inspired by the above-mentioned works, we investigate the centre of the quantum group $U=U_{q,G}(\mathfrak{g})$ ($n\geqslant2$, otherwise it is $U_{q,q^{-1}} (\mathfrak{sl}_2)$), under the simplest assumption that all parameters in $ \{q\} \cup \{q_{ij} \mid i < j\} $ are algebraically independent, and prove that the Harish-Chandra homomorphism $\xi:Z(U) \rightarrow U^0$ is injective, and the image $\operatorname{Im}(\xi) = (U^0_\flat)^W$ for all types. The key difference in our approach arises from the fact that $ U_{q,G}(\mathfrak{g}) $ involves more algebraically independent parameters than $ U_{r,s}(\mathfrak{g}) $. These parameters induce stronger conditions (see Lemma \ref{nondeg of rho^{-,-}}), which are essential in guaranteeing $\operatorname{Im}(\xi)\subseteq (U^0_\flat)^W$ for all types.
	The isomorphism $\xi: Z(U) \rightarrow (U_\flat^0)^W$ indicates $Z(U)=\operatorname{span}_{\mathbb{K}}\{z_\lambda\;\big|\;\lambda\in\Psi:= \Lambda^+ \cap Q\}$, which confirms that the Batra-Yamane conjecture \cite[Conj.~3.13]{BY20} holds for $Z(U_{q,G})$ under the given conditions (see Corollary \ref{conj}).
	
	Although the set $\{z_\lambda \mid \lambda \in \Psi_{\min}\}$ generates $Z(U)$, where $\Psi_{\min}$ is the minimal generating set of the monoid $\Psi$, the algebraic relations among them depend on Steinberg's formula.
	To obtain a more direct formulation of $Z(U)$, motivated by Li-Xia-Zhang's work \cite{LXZ16} on $U_q(\mathfrak{g})$ (where their $\Psi':=\Lambda^+ \cap \frac{Q}{2}$), we shall construct an algebra monomorphism $\theta: (U^0_\flat)^W \rightarrow {\mathbb{K}[\mathfrak{h}^*]^W}$ and find another generating set of $\operatorname{Im}(\theta)$ whose multiplicative relations are completely determined by the additive relations among $\Psi_{\min}$ in $\Psi$.
	So we can describe the $Z(U)\cong(U^0_\flat)^W\cong \operatorname{Im}(\theta)$ by (a quotient of) a polynomial algebra with the explicit generators (and generating relations). Our main results can be summarized as follows:
	
	\begin{theorem} \label{main results}
		Let $U=U_{q,G}(\mathfrak{g})\; (n\geqslant 2)$, we have the following Harish-Chandra theorem:
		\begin{enumerate}
			\item The Harish-Chandra homomorphism $\xi:\; Z(U) \rightarrow U^0$ is injective. The image $\xi(Z(U)) = (U^0_\flat)^W$.
			
			\item The diagram
			\begin{equation*}
				\begin{tikzcd}[column sep=small]
					R(\mathfrak{g}) \otimes_\mathbb{Z} \mathbb{K} \arrow[d,"\textnormal{Ch}"', "\wr"] \arrow[r, dashed] & {Z(U_{q,G}(\mathfrak{g}))} \arrow[d, "\xi"', "\wr"] & {[L(\lambda)]} \arrow[r, "\lambda \in \Lambda^+\cap Q",dashed, maps to] \arrow[d, maps to] & z_\lambda \arrow[d, maps to]                                                                       \\
					{\mathbb{K}[\mathfrak{h}^*]^W} &
					(U^0_\flat)^W \arrow[l, "\theta"', hook']                    & {\sum\limits_{\mu \leqslant\lambda} \dim(L(\lambda)_\mu) \, e^\mu}  & {\sum\limits_{\mu \leqslant\lambda} \dim(L(\lambda)_\mu) \, \omega'_\mu \omega_{-\mu}} \arrow[l, maps to]
				\end{tikzcd}
			\end{equation*}
			
			commutes, where $R(\mathfrak{g})$ denotes the Green ring of $\mathfrak{g}$ , and the dashed arrow represents a partial map only defined on $[L(\lambda)],\; \lambda \in \Psi = \Lambda^+ \cap Q$.
			
			\item The centre $Z(U)\cong(U^0_\flat)^W$ is isomorphic to a polynomial algebra $S=\mathbb{K}[t_1, \cdots, t_n]$ when $\mathfrak{g}$ is of type
			$B_n$, $E_8$, $F_4$ or $G_2$, otherwise it is isomorphic to a quotient  $S/I$, where $S$ is a polynomial algebra of rank $|\Psi_{\min}|$, and $I$ is	 the relation ideal given in Theorem \ref{Last theorem}.
			
			\item With the same monoid $\Psi=\Lambda^+\cap Q$, the structure of $(U^0_\flat)^W$ in Proposition \ref{HW25}.1 for the two-parameter quantum group $U_{r,s}(\mathfrak{g})$ matches Theorem \ref{main results}.3 identically. That is,
			\begin{equation*}
				\xi(Z(U_{r,s}))
				\begin{cases}	
					= (U^0_\flat)^W   \hspace{1.9cm} \cong  S/I, &\textnormal{if $n$ is even,}\\
					\supseteq (U^0_\flat)^W \otimes \mathbb{K}[z_{*},z_{*}^{-1}] \cong S/I \otimes \mathbb{K}[z_{*},z_{*}^{-1}],
					&\textnormal{if $n$ is odd.}
				\end{cases}
			\end{equation*}
		\end{enumerate}
	\end{theorem}
	
	The paper is organized as follows.
	In Section 2, we present the definition of the quantum groups $U_{q,G}(\mathfrak{g})$, the Rosso form, and the weight module.
	In Section 3, we introduce the Harish-Chandra homomorphism $\xi$, prove its injectivity, and show that $\operatorname{Im}(\xi) \subseteq (U_\flat^0)^W$. In Section 4, we construct central elements $z_\lambda$ by realizing the quantum trace $t_\lambda$ on the weight module $L(\lambda)$ via the Rosso form. We prove that the images $\xi(z_\lambda)$ span the subalgebra $(U_\flat^0)^W$, thereby establishing $\operatorname{Im}(\xi) \supseteq (U_\flat^0)^W$. Consequently, the proof of Theorem \ref{main results}.1 is finished, and then $Z(U)$ has the basis $\{z_\lambda\;\big|\;\lambda\in\Psi\}$, which supports the Batra-Yamane conjecture in our case. In Section 5, we construct an algebra monomorphism $ \theta: (U^0_\flat)^W \rightarrow {\mathbb{K}[\mathfrak{h}^*]^W} $, making the diagram in Theorem \ref{main results}.2 commute. Then we provide minimal generating sets for both $\Psi$ and $\operatorname{Im}(\theta)$. Finally, we construct an epimorphism $\phi$ from a rank-$|\Psi_{\min}|$ polynomial algebra $S$ onto $\operatorname{Im}(\theta)$, and determine a generating set for the ideal $I = \ker(\phi)$. Then it turns out that $(U^0_\flat)^W\cong\operatorname{Im}(\theta)\cong S/I$, which completes the proofs of Theorems \ref{main results}.3 and \ref{main results}.4.
	
	\section{Preliminaries.}
	
	\subsection{Specialized multi-parameter quantum groups and the Rosso form}
	
	Let $I$ be a set of integers and $A=[a_{ij}]_{i,j\in I}$ be a symmetrizable Cartan matrix corresponding to a complex simple Lie algebra $\mathfrak{g}$. Let $\left\{ d_i \right\} _{i\in I}$ be positive integers such that $d_ia_{ij}=d_ja_{ji}$ and $\operatorname{gcd}\{d_i\}_{i\in I}=1$.
	Let $\mathbb{K}$ be an algebraically closed field of characteristic zero, $q\in \mathbb{K}^{\times}$, and $G=[q_{ij}]_{i,j\in I}$ be a matrix such that $q_{ij}\in \mathbb{K}^{\times},q_{ii}=1$ and $q_{ij}q_{ji}=1$. We assume that the parameters in the set $ \{q\} \cup \{q_{ij} \mid i < j\} $ are algebraically independent (certainly this excludes the root of unity cases).
	
	\begin{definition} \cite{H92,OY90}
		Let $U=U_{q,G}(\mathfrak{g})$ be the unital associative algebra over $\mathbb{K}$ generated by elements $e_i$, $f_i$, $\omega_i^{\pm 1}$, $\omega_i^{\prime\pm 1}$ satisfying:
		\begin{equation*}
			\begin{split}
				\textnormal{(A1)} & \hspace{1cm} \omega_i \omega_i^{-1} =
				\omega'_i {\omega'_i}^{-1} = 1,
				\quad \\
				& \hspace{1cm}
				\omega_i \omega_j = \omega_j \omega_i, \quad
				\omega'_i \omega'_j = \omega'_j \omega'_i, \quad
				\omega'_i \omega_j = \omega_j \omega'_i,\\
				\textnormal{(A2)} & \hspace{1cm}
				\omega_i e_j \omega_i^{-1} = q_i^{a_{ij}} q_{ij} e_j, \hspace{0.7cm}
				\omega'_i e_j {\omega'_i}^{-1} = q_j^{-a_{ji}} q_{ji}^{-1} e_j,\\
				& \hspace{1cm}
				\omega_i f_j \omega_i^{-1} = q_i^{-a_{ij}} q_{ij}^{-1} f_j, \quad
				\omega'_i f_j {\omega'_i}^{-1} = q_j^{a_{ji}} q_{ji} f_j,\\
				\textnormal{(A3)} & \hspace{1cm}
				e_i f_j - f_j e_i = \delta_{ij} \cdot \frac{\omega_i - \omega_i^{\prime}}{q_i - q_i^{-1}},\\
				\textnormal{(A4)} & \hspace{1cm}
				\sum_{n=0}^{1 - a_{ij}} (-q_{ij})^n
				\begin{bmatrix}
					1 - a_{ij} \\ n
				\end{bmatrix}_i
				e_i^{1 - a_{ij} - n} e_j e_i^n = 0, \quad (i \ne j),\\
				& \hspace{1cm}
				\sum_{n=0}^{1 - a_{ij}} (-q_{ij})^n
				\begin{bmatrix}
					1 - a_{ij} \\ n
				\end{bmatrix}_i
				f_i^{1 - a_{ij} - n} f_j f_i^n = 0, \quad (i \ne j),
			\end{split}
		\end{equation*}
		where $q_i=q^{d_i}$, $i\in I$, and $\left[ \begin{smallmatrix} m \\ n \end{smallmatrix} \right]_i$ is defined by
		\begin{equation*}
			\left[ \begin{matrix} m \\ n \end{matrix} \right]_i
			= \frac{[m]_i [m-1]_i \cdots [m - n + 1]_i}{[n]_i [n-1]_i \cdots [1]_i}, \qquad
			[m]_i = \frac{q_i^m - q_i^{-m}}{q_i - q_i^{-1}}.
		\end{equation*}
	\end{definition}
	
	The algebra $U=U_{q,G}$ constitutes a multi-parameter quantum enveloping algebra (see \cite{PHR10} for general definition), equipped with the Hopf algebra structure $(U, \Delta, \varepsilon, S)$ defined by
	\begin{equation*}
		\begin{split}
			\textnormal{(C1)} & \hspace{1cm}
			\Delta(\omega_i) = \omega_i \otimes \omega_i, \hspace{1.55cm}
			\Delta(\omega'_i) = \omega'_i \otimes \omega'_i,\\
			& \hspace{1cm}
			\Delta(e_i) = e_i \otimes 1 + \omega_i \otimes e_i, \quad
			\Delta(f_i) = f_i \otimes {\omega'_i}^{-1} + 1 \otimes f_i,\\
			\textnormal{(C2)} & \hspace{1cm}
			\varepsilon(\omega_i) = \varepsilon(\omega'_i) = 1, \hspace{2.15cm}
			\varepsilon(e_i) = \varepsilon(f_i) = 0,\\
			\textnormal{(S1)} & \hspace{1cm}
			S(\omega_i) = \omega_i^{-1}, \quad
			S(\omega'_i) = {\omega'_i}^{-1}, \quad
			S(e_i) = -\omega_i^{-1} e_i, \quad
			S(f_i) = -f_i \omega'_i.
		\end{split}
	\end{equation*}
	
	Refer to \cite[Thm.~20]{PHR10}, there exists a unique bilinear pairing $\left<-,-\right> $: $U^{\leqslant 0} \times U^{\geqslant 0}
	\rightarrow \mathbb{K}$ such that for all $x$, $x^{\prime}\in U^{\leqslant 0}$, $y$, $y^{\prime} \in U^{\geqslant 0}$, $\mu$, $\nu \in Q$, and $i$, $j \in I$
	\begin{equation*}
		\begin{split}
			& \langle y, xx' \rangle = \langle \Delta(y), x' \otimes x \rangle, \quad
			\langle yy', x \rangle = \langle y \otimes y', \Delta(x) \rangle, \\
			& \langle f_i, e_j \rangle = \delta_{ij} \cdot \frac{1}{q_i^{-1} - q_i}, \\
			& \langle \omega'_i, \omega_j \rangle = q_j^{a_{ji}} q_{ji}, \quad
			\langle \omega_i^{\prime \pm 1}, \omega_j^{-1} \rangle = \langle \omega_i^{\prime \pm 1}, \omega_j \rangle^{-1} = \langle \omega'_i, \omega_j \rangle^{\mp 1},
		\end{split}
	\end{equation*}
	for $1 \leqslant i,j \leqslant n$, and all other pairs of generators are $0$.
	
	\begin{theorem}\label{trangular decom}\cite[Cor.~22]{PHR10}
		The algebra $U=U_{q,G}(\mathfrak{g})$ has the triangular decomposition
		$U \cong  U_{q,G}  \left( \mathfrak{n} ^- \right) \otimes U^0\otimes U_{q,G}\left( \mathfrak{n} \right)$,
		abbreviated as $U=U^{-}U^{0}U^{+}$, where $\mathfrak{g}=\mathfrak{n} ^-\oplus \mathfrak{h} \oplus \mathfrak{n}$ is the semisimple Lie algebra corresponding to the Cartan matrix $A$.
	\end{theorem}
	
	The algebra $U$ is $Q$-graded given by
	\begin{equation*}
		\deg e_i = \alpha_i, \quad \deg f_i = -\alpha_i, \quad \deg \omega_i = \deg \omega_i^\prime = 0,
	\end{equation*}
	it has $Q^{+}$-graded subalgebras
	\begin{equation*}
		U^{\pm}
		= \bigoplus_{\mu \in Q^+} U_{\pm\mu}^\pm
		= \bigoplus_{\mu \in Q^+}
		\left\{ x \in U^{\pm} \;\middle|\;
		\begin{aligned}
			&\omega_{\eta} x \omega_{\eta}^{-1} = \langle \omega_{\pm\mu}^{\prime}, \omega_{\eta} \rangle x \\
			&\omega_{\eta}^{\prime} x \omega_{\eta}^{\prime -1} = \langle \omega_{\eta}^{\prime}, \omega_{\pm\mu} \rangle^{-1} x
		\end{aligned}
		,\; \eta \in Q
		\right\}.
	\end{equation*}
	One can also define a homomorphism
	$\varrho^{\lambda}:U^{0}\rightarrow \mathbb{K}$ for each $\lambda=\sum\nolimits_{i=1}^n {\lambda _i\alpha _i} \in \Lambda$
	satisfying
	\begin{equation*}
		\begin{split}
			\varrho^{\lambda}(\omega_j)
			&= \prod_{i=1}^n \left\langle \omega_i', \omega_j \right\rangle^{\lambda_i}
			= \prod_{i=1}^n \left( q_j^{a_{ji}} q_{ji} \right)^{\lambda_i},\\
			\varrho ^{\lambda}\left( \omega _{j}^{\prime} \right)
			&= \prod_{i=1}^n{\left< \omega _{j}^{\prime},\omega _{i}\right>}^{-\lambda_i}=\prod_{i=1}^n{\left( q_{i}^{a_{ij}}q_{ij} \right) ^{-\lambda_i}}.
		\end{split}
	\end{equation*}
	
	The following lemma is the key in proving the injectivity of the Harish-Chandra homomorphism.
	\begin{lemma} \label{iff zero}
		Let $\rho$ be the half sum of positive
		roots. Given any $\eta$, $\phi$ $\in Q$ and any integer $N$, we have
		$\varrho^{\lambda}(\omega _{\eta}^{\prime}\omega _{\phi})=1$ for any $\lambda \in \Lambda^{+} $ with $\lambda \geqslant$ $N\rho$, if and only if $(\eta,\phi)=(0,0)$.
	\end{lemma}
	\begin{proof}
		It suffices to prove the necessity. Fix elements $\eta, \phi \in Q$ and $\lambda \in \Lambda^{+}$, viewing all three as column vectors in the basis $\left\{ \alpha _i \right\} _{i=1}^{n}$, one has
		\begin{equation*}
			\begin{split}
				\varrho^{\lambda}\left( \omega_{\phi} \right)
				&= \prod_{j=1}^n \left( \varrho^{\lambda}(\omega_j) \right)^{\phi_j} \hspace{0.3cm}
				= \prod_{i,j=1}^n \left( q_j^{a_{ji}} q_{ji} \right)^{\lambda_i \phi_j} \\
				&= q^{\sum_{i,j} \phi_j d_j a_{ji} \lambda_i} \left( \prod_{i,j=1}^n q_{ji}^{\lambda_i \phi_j}  \right)
				= q^{\phi^T D A \lambda} \left(\prod_{i,j=1}^n q_{ij}^{-\lambda_i \phi_j} \right).
			\end{split}
		\end{equation*}
		Similarly, one has
		$\varrho ^{\lambda}\left( \omega _{\eta}^{\prime} \right)
		= q^{-\eta^T DA\lambda} \left( {\prod_{i,j}{q_{ij}^{-\lambda_i\eta_j}}} \right) $.
		Hence, the condition
		$\varrho ^{\lambda}\left( \omega _{\eta}^{\prime}\omega _{\phi} \right)$ $=1$ is equivalent to
		$\left( \phi-\eta \right)^T DA\lambda =0$, and $-\lambda _i\left( \eta +\phi  \right)_j +\lambda _j\left( \eta +\phi \right)_i=0 $
		for all $\lambda \in \Lambda^{+}$, $\lambda \geqslant N\rho$,
		which easily leads to $(\eta,\phi)=(0,0)$.
	\end{proof}

	The proof of $\xi(Z(U)) \supseteq (U^0_\flat)^W$ for the Harish-Chandra theorem requires central elements constructed via the following quantum Killing form.
	\begin{definition}
		The bilinear form $\langle -
		\mid-\rangle $: $U \times U \rightarrow \mathbb{K}$ defined by
		\begin{equation*}
			\begin{split}
				\left\langle (y_1 \omega_{\nu_1}^\prime) \, \omega_{\eta_1}^\prime \omega_{\phi_1} x_1 \mid (y_2 {\omega_{\nu_2}^\prime}^{-1}) \, \omega_{\eta_2}^\prime \omega_{\phi_2} x_2 \right\rangle & =
				\langle y_2, x_1 \rangle \langle \omega_{\eta_2}^\prime, \omega_{\phi_1} \rangle \langle \omega_{\eta_1}^\prime, \omega_{\phi_2} \rangle \langle S^2(y_1), x_2 \rangle \\
				& = q^{2(\rho, \nu_1)} \langle y_2, x_1 \rangle \langle \omega_{\eta_2}^\prime, \omega_{\phi_1} \rangle \langle \omega_{\eta_1}^\prime, \omega_{\phi_2} \rangle \langle y_1, x_2 \rangle
			\end{split}
		\end{equation*}
		is called the Rosso form of $U$, where
		$x_i \in U^+_{\mu_i},\; y_i \in U^-_{-\nu_i},\; \mu_i, \nu_i, \phi_i, \eta_i \in Q^+$.
	\end{definition}
	
	\begin{proposition}\cite{PHR10}\label{graded Rosso}
		As a specialized multi-parameter quantization, the algebra $U=U_{q,G}$ admits:
		\begin{enumerate}
			\item The Rosso form is $\operatorname{ad}_l$-$invariant$, i.e., for all $a,b,c \in U$, one has
			\begin{equation*}
				\langle \,\operatorname{ad}_l\left( a \right) b
				\mid c\rangle =\langle b    \mid \operatorname{ad}_l\left( S\left( a \right) \right) c\,\rangle,
			\end{equation*}
			\item The Rosso form satisfies the orthogonality condition:
			if $\mu_i,\nu_i \in Q^+$, then
			$\langle U_{-\nu _1}^{-}U^0U_{\mu _1}^{+}\mid U_{-\nu _2}^{-}U^0U_{\mu _2}^{+}\rangle$ $\ne 0$
			if and only if $\mu_1=\nu_2,\nu_1=\mu_2$.
		\end{enumerate}
	\end{proposition}

	\begin{theorem}\label{Rosso nondeg}
		The Rosso form $\langle -\mid-\rangle$ of $U$ is nondegenerate.
	\end{theorem}
	\begin{proof}
		Since the pair $\langle -,-\rangle $ has orthogonality for the grading, it suffices to check the case when $u \in U_{-\nu}^{-}U^0U_{\mu}^{+}$. If $\langle u\mid v\rangle=0 $ holds for all $v\in U_{-\mu}^{-}U^0U_{\nu}^{+}$, then $u=0$.
		
		Denote $d_\mu = \dim U_{\mu}^{+}$. Let $\left\{ u_{1}^{\mu},...,u_{d_{\mu}}^{\mu} \right\} $ be a basis of $U_\mu^+$ and $\left\{ v_{1}^{\mu},...,v_{d_{\mu}}^{\mu} \right\} $ be its dual basis in $U^-_{-\mu}$ with respect to the pair, that is, $
		\left< v_{i}^{\mu},u_{j}^{\mu} \right> =\delta _{ij}
		$. Hence
		\begin{equation*}
			U_{-\nu}^{-} U^0 U_{\mu}^{+} = \operatorname{span}_{\mathbb{K}} \left\{ \left. \left( v_i^{\nu} {\omega_{\nu}^\prime}^{-1} \right) \omega_{\eta}^\prime \omega_{\phi} u_j^{\mu} \ \right| \ 1 \leqslant i \leqslant d_{\nu},\ 1 \leqslant j \leqslant d_{\mu} \right\}.
		\end{equation*}
		
		Let $u=\sum_{i,j,\eta, \phi}{k_{i,j,\eta, \phi}\left( v_{i}^{\nu}{\omega _{\nu}^{\prime}}^{-1} \right) \omega _{\eta}^{\prime}\omega _{\phi}u_{j}^{\mu}}$,
		and $v=\left( v_{k}^{\mu}{\omega _{\mu}^{\prime}}^{-1} \right) \omega _{\eta \prime}^{\prime}\omega _{\phi \prime}u_{l}^{\nu}$,
		where $1\leqslant k\leqslant d_{\mu},1\leqslant l\leqslant d_{\nu},\eta^\prime,\phi^\prime\in Q$.
		Suppose $\langle u
		\mid v\rangle =0$, we have
		\begin{equation} \label{Dedekind}
			\begin{aligned}
				0 &= \sum_{\eta, \phi} k_{k,l,\eta,\phi} \, q^{2(\rho, \nu)} \, \langle \omega_{\eta}^\prime, \omega_{\phi'} \rangle \, \langle \omega_{\eta'}^\prime, \omega_{\phi} \rangle \\
				&= \sum_{\eta, \phi} k_{k,l,\eta,\phi} \, q^{2(\rho, \nu)} \, \chi_{\eta,\phi}(\eta', \phi'),
			\end{aligned}
		\end{equation}
		where $\chi _{\eta ,\phi}: Q\times Q\rightarrow \mathbb{K} ^{\times}$ denotes the group homomorphism by
		$\chi _{\eta ,\phi}\left( \eta^\prime,\phi^\prime \right) =\langle \omega _{\eta}^{\prime},\omega _{\phi \prime}\rangle \langle \omega _{\eta \prime}^{\prime},\omega _{\phi}\rangle$.
		
		We claim that these $\chi _{\eta ,\phi}$ in Eq.(\ref{Dedekind}) are different characters, i.e., if $\chi _{\eta, \phi}=\chi _{\eta \prime,\phi \prime}$, then $(\eta,\phi)=(\eta^\prime,\phi^\prime)$.
		This is verified by
		\begin{equation*}
			1 = \frac{\chi_{\eta, \phi}(0, \alpha_j)}{\chi_{\eta^\prime, \phi^\prime}(0, \alpha_j)}
			= \prod_{i=1}^n \left< \omega_i^\prime, \omega_j \right>^{\eta_i - \eta_i^\prime}
			= \prod_{i=1}^n a_{ji}^{\eta_i - \eta_i^\prime}
			= \prod_{i=1}^n \left( q_j^{a_{ji}} q_{ji} \right)^{\eta_i - \eta_i^\prime},
		\end{equation*}
		which in turn gives $DA(\eta-\eta^\prime)=0$, then $\eta=\eta^\prime$, since $DA$ is invertible; a similar argument leads to the conclusion $\phi=\phi^\prime$.
		
		Finally, we have $k_{l,k,\eta,\phi}=0$ from Eq.(\ref{Dedekind}) by Dedekind theorem, which leads to $u=0$.
	\end{proof}
	
	\subsection{Weight modules of $U_{q,G}$}
	
	Let $\mathcal{B}= U_{q,G}\left( \mathfrak{b} \right)$, and $V^\psi$ be the one-dimensional $\mathcal{B}$-module, on which each $e_i$ acts as multiplication by $0$, and $U^0$ acts via an algebra homomorphism $\psi:U^0 \rightarrow \mathbb{K}$. By Theorem \ref{trangular decom}, the Verma module $M(\psi)$ with highest weight $\psi$ is induced from $V^{\psi}$ as $M(\psi)=U_{q,G}\otimes_{\mathcal{B}}V^{\psi}$.
	
	\begin{proposition}\cite{PHR10}\label{irr mod}
		For $U=U_{q,G}$, we have
		\begin{enumerate}
			\item Let $ v_\lambda $ be a highest weight vector of $ M(\lambda) $ for $ \lambda \in \Lambda^+ $, then
			\begin{equation*}
				L(\lambda) = M(\lambda)\Big/ \left( \sum_{i=1}^n U f_i^{( \lambda, \alpha_i^\vee ) + 1} \cdot v_\lambda \right)
			\end{equation*}
			is simple, with the decomposition of the weight space
			$L(\lambda) = \bigoplus_{\eta \leqslant \lambda} L(\lambda)_\eta,$
			where
			\begin{equation*}
				L(\lambda)_\eta = \{\, x \in L(\lambda) \mid \omega_i \cdot x = \varrho^\eta(\omega_i) x, \; \omega'_i \cdot x = \varrho^\eta(\omega'_i) x, \; \forall\; i\in I\, \}.
			\end{equation*}
			\item The elements $ e_i, f_i, \; \forall\; i\in I $ act locally nilpotently on $ L(\lambda) $.
			\item $\dim L(\lambda)_\eta = \dim L(\lambda)_{\sigma(\eta)}, \; \forall\; \eta \in \Lambda, \; \sigma \in W$.
			\item For any $\beta=\sum_{i=1}^n{m_i\alpha _i}\in Q^+$ such that all $m_i\leqslant \left( \lambda, \alpha _{i}^{\lor} \right) $, the linear map $x \mapsto x.v_{\lambda}$ from $U_{-\beta}^{-}$ to $M(\lambda)$ is injective.
		\end{enumerate}
	\end{proposition}

	\section{Injectivity and the image of Harish-Chandra homomorphism $\xi$}
	
	\subsection{Injective properties} \label{sec3.1}
	
	Let $Z(U)$ be the centre of $U$, then $Z(U)\subset U_{0}$. Defining an algebra homomorphism
	$\gamma ^{-\rho}:U^0\rightarrow U^0$ as
	$\gamma ^{-\rho}(\omega _{\eta}^{\prime}\omega _{\phi})=\varrho ^{-\rho}\left( \omega _{\eta}^{\prime}\omega _{\phi} \right) \omega _{\eta}^{\prime}\omega _{\phi}$,
	and the canonical projection $\pi : U_0 \rightarrow U^0$.
	\begin{definition}
		The Harish-Chandra homomorphism $\xi : Z(U) \rightarrow U^0$ is defined by
		\begin{equation*}
			\xi = \gamma^{-\rho} \circ \left. \pi \right|_{Z(U)},
		\end{equation*}
		and the Weyl group $W$ has a natural action on $U^0_\flat := \bigoplus_{\eta \in Q} \mathbb{K} \, \omega'_\eta \omega_{-\eta}$ that
		\begin{equation*}
			\sigma(\omega'_\eta \omega_{-\eta}) := \omega'_{\sigma(\eta)} \omega_{-\sigma(\eta)}, \quad \forall\; \sigma \in W, \; \eta \in Q.
		\end{equation*}
	\end{definition}
	
	\begin{theorem}\label{injective}
		The Harish-Chandra homomorphism $\xi:Z(U)\rightarrow U^0$ is injective.
	\end{theorem}
	
	\begin{proof}
		We aim to show that if $z \in Z(U)$ satisfies $\xi(z) = 0$, then $z = 0$.
		
		Since $Z(U) \subset U_0 = U^0 \bigoplus K$, where $K$ is the two-sided ideal $\bigoplus_{\nu > 0} U^-_{-\nu} U^0 U^+_{\nu}$ and is the kernel of $\xi$, we write $z = \sum_{\nu \in Q^+} z_\nu$ with each $z_\nu \in U^-_{-\nu} U^0 U^+_\nu$. The condition $\xi(z) = 0$ implies $z_0 = 0$. Now assume that $z \ne 0$, then there exists a minimal $\nu \in Q^+$ such that $z_\nu \ne 0$.
		
		Choose bases $\{x_l\}$ and $\{y_k\}$ for $U^+_\nu$ and $U^-_{-\nu}$, respectively. Writing $z_\nu= \sum_{k,l} y_k t_{k,l} x_l, \; t_{k,l} \in U^0$, we have for each $i$:
		\begin{equation*}
			0 = e_i z - z e_i = \sum_{\gamma \ne \nu} (e_i z_\gamma - z_\gamma e_i) + (e_i z_\nu - z_\nu e_i).
		\end{equation*}
		For the last term
		\begin{equation*}
			e_i z_\nu - z_\nu e_i = \sum_{k,l} (e_i y_k - y_k e_i) t_{k,l} x_l + \sum_{k,l} y_k (e_i t_{k,l} x_l - t_{k,l} x_l e_i),
		\end{equation*}
		we have $\sum_{k,l} (e_i y_k - y_k e_i) t_{k,l} x_l = 0$ since this term alone lies in $U^-_{-(\nu - \alpha_i)} U^0 U^+_\nu$.
		Hence $\sum_k e_i y_k t_{k,l} = \sum_k y_k e_i t_{k,l}$ for each $l$ since $\{x_l\}$ forms a basis of $U^+_\nu$.
		Fixing an index $l$, for each $\lambda \in \Lambda^+$, we claim that
		\begin{equation*}
			m := \left( \sum_k y_k t_{k,l} \right) \cdot v_\lambda = 0
		\end{equation*}
		in the irreducible weight module $L(\lambda)$ with the highest weight vector $v_\lambda$.
		Indeed, the identities
		\begin{equation*}
			e_i m = \sum_k e_i y_k t_{k,l} \cdot v_\lambda = \sum_k y_k e_i t_{k,l} \cdot v_\lambda = 0, \quad \forall\; i,
		\end{equation*}
		imply that $m$ is a weight vector annihilated by all $e_i$ and lies in a lower weight space, forcing $m=0$ to prevent generating a proper submodule of $L(\lambda)$.
		It follows that
		\begin{equation*}
			\sum_k y_k \varrho^\lambda(t_{k,l}) v_\lambda = 0.
		\end{equation*}
		Taking $N=\max \left\{ \left. k_i \right| \nu =k_1\alpha _1+\cdots+k_n\alpha _n \right\}$ and for any $\lambda \geqslant N\rho $,
		we have $\left( \lambda, {\alpha _i}^{\lor} \right) \geqslant k_i$ and the injective map
		$
		\sum_k y_k \varrho^\lambda(t_{k,l}) \mapsto \left( \sum_k y_k \varrho^\lambda(t_{k,l}) \right) \cdot v_\lambda
		$
		through Lemma \ref{irr mod}, then conclude that
		$
		\sum_k \varrho^\lambda(t_{k,l}) y_k  = 0,
		$
		which leads to $\varrho^\lambda(t_{k,l}) = 0$ for each $k$.
		
		Next, write each $t_{k,l}$ as a finite sum:
		$
		t_{k,l} = \sum_{\eta, \phi} k_{\eta,\phi} \omega'_\eta \omega_\phi, \; k_{\eta,\phi} \in \mathbb{K}.
		$
		Then for all $m \in \mathbb{Z}_{> 0}$, we have
		\begin{equation*}
			0 = \varrho^{m\lambda}(t_{k,l}) = \sum_{\eta,\phi} \left( \varrho^\lambda(\omega'_\eta \omega_\phi) \right)^m k_{\eta,\phi}.
		\end{equation*}
		This gives a homogeneous linear system with a Vandermonde matrix since the values $\varrho^\lambda(\omega'_\eta \omega_\phi)$ are pairwise distinct by Proposition \ref{iff zero}. Therefore, the only solution is $k_{\eta,\phi} = 0$ for all $\eta,\phi$, implying $t_{k,l} = 0$.
		Hence $z_\nu = 0$, contradicting the assumption $z_\nu \ne 0$. This contradiction implies that $z = 0$.
	\end{proof}
	
	\subsection{The image $\operatorname{Im}(\xi)$}
	
	Let $\lambda, \mu \in \Lambda$. Define an algebra map $\varrho^{\lambda,\mu}: U^{0} \to \mathbb{K}$ by $\varrho^{\lambda,\mu}=\varrho^{0,\lambda}\varrho^{\mu,0}$, where
	\begin{equation*}
		\varrho^{0,\mu} :\ \omega_{\eta}^\prime \omega_{\phi} \mapsto  q^{(\eta + \phi, \mu)}, \quad
		\varrho^{\lambda,0} :\ \omega_{\eta}^\prime \omega_{\phi} \mapsto \varrho^{\lambda} \left( \omega_{\eta}^\prime \omega_{\phi} \right).
	\end{equation*}

	\begin{lemma} \label{nondeg of rho^{-,-}}
		We have the following properties:
		\begin{enumerate}
			\item Let $u=\omega _{\eta}^{\prime}\omega _{\phi}$, $\eta$, $\phi \in Q$. If
			$\varrho^{\lambda,\mu}(u)=1$ for all $\lambda,\mu \in \Lambda$, then $u=1$;
			\item If $u\in U^{0}$ satisfies $\varrho^{\lambda,\mu}(u)=0$ for all $\lambda,\mu \in \Lambda$, then $u=0$.
		\end{enumerate}
	\end{lemma}
	
	\begin{proof}		
		(1) Since
		$
		1=\varrho^{\lambda,\mu}(\omega^{\prime}_\eta \omega_{\phi})=
		q^{(\phi-\eta)^T D A \lambda} \left( \prod_{i,j} q_{ij}^{-\lambda_i (\phi + \eta)_j} \right)
		\cdot
		q^{(\eta + \phi, \mu)},
		$
		the algebraic independence of the parameters $q$ and $\{q_{ij}\}$
		implies the identity
		\begin{equation} \label{strong equation}
			\begin{cases}
				(\phi - \eta)^T D A \lambda + (\eta + \phi, \mu) = 0, \\
				-\lambda _i\left( \eta +\phi \right)_j +\lambda _j\left( \eta +\phi \right)_i
				=0.
			\end{cases}
		\end{equation}
		When $\lambda=0$, the equality $\eta+\phi=0$ holds. Setting $\mu=0$, the invertibility of $DA$ implies $\eta-\phi=0$. Consequently, we conclude $\eta=\phi=0$ and $u=1$.
		
		(2) Fixing a pair $(\eta,\phi)\in Q\times Q$, one can define a character $\kappa _{\eta,\phi}$ on the group $\Lambda \times \Lambda$ to be $\kappa_{\eta,\phi}:(\lambda,\mu)\mapsto \varrho^{\lambda,\mu}(\omega^{\prime}_{\mu}\omega_{\phi})$. Let $u=\sum_{\left( \eta, \phi \right)}{k_{\eta, \phi}\omega _{\eta}^{\prime}\omega _{\phi}},k_{\eta, \phi} \in \mathbb{K}$.
		Then
		\begin{equation*}
			0 = \varrho^{\lambda,\mu}(u) = \sum_{(\eta, \phi)} k_{\eta, \phi} \varrho^{\lambda,\mu} \left( \omega_{\eta}^\prime \omega_{\phi} \right) = \sum_{(\eta, \phi)} k_{\eta, \phi} \kappa_{\eta, \phi} \left( \omega_{\eta}^\prime \omega_{\phi} \right).
		\end{equation*}
		The property (1) implies that all characters ${\kappa _{\eta, \phi}}$ are distinct, which forces all $k_{\eta,\phi}=0$ and $u=0$.
	\end{proof}
	
	\begin{remark}
		This result differs from \cite[Lem.~23]{HW25} in the two-parameter quantum groups $U_{r,s}(\mathfrak{g})$, 
		where $\mathfrak{g}$ is of even rank.
		The reason is that $U_{q,G}(\mathfrak{g})$ involves more algebraically independent parameters $q$ and $\{q_{ij}\}$, and provides more restrictive conditions (\ref{strong equation}) on weights $\phi$ and $\eta$, and further guarantees $\xi(Z(U)) \subseteq (U^0_\flat)^W$ for the Harish-Chandra homomorphism for any rank in the following process.
	\end{remark}
	
	\begin{lemma}\label{weyl group act}
		$\varrho^{\sigma(\lambda),\mu}(u)=\varrho^{\lambda,\mu}(\sigma^{-1}(u))$, for $u\in U^0_\flat$, $\sigma \in W$, $\lambda,\mu \in \Lambda$.
	\end{lemma}
	
	\begin{proof}
		It suffices to check when $u=\omega _{\eta}^{\prime}\omega _{-\eta}$ and $\sigma$ is a simple reflection. Let $\Sigma$ denote the matrix of $\sigma$, we have
		\begin{equation*}
			\begin{split}
				\varrho^{\lambda,0} \big(\sigma_i^{-1}(u)\big)
				&= \varrho^{\lambda} \big(\omega_{\sigma(\eta)}' \omega_{-\eta} \big)= q^{-2 (\sigma(\eta))^T D A \lambda} = q^{-2 \eta^T \Sigma D A \lambda} \\
				&= q^{-2 \eta^T D A \Sigma \lambda} = q^{-2 \eta D A \sigma(\lambda)} = \varrho^{\sigma(\lambda),0} \big(\omega_{\eta}' \omega_{-\eta} \big), \\
				\varrho^{0,\mu} \big( \sigma^{-1}(u) \big)
				&= q^{ \left( \sigma^{-1}(\eta) + \sigma^{-1}(-\eta), \mu \right)}
				= q^{(0, \mu)} = \varrho^{0,\mu} \left( \omega_{\eta}' \omega_{-\eta} \right).
			\end{split}
		\end{equation*}
		
		This completes the proof.
	\end{proof}

	\begin{lemma}\label{weyl group inv}
		For $\sigma \in W$, $\lambda,\mu \in \Lambda$. If $u\in U^0$ satisfies $\varrho ^{\sigma \left( \lambda \right), \mu}\left( u \right) =\varrho ^{\lambda, \mu}\left( u \right)$,
		then $u\in (U^0_\flat)^W$.
	\end{lemma}
	
	\begin{proof}
		Let $u=\sum_{\left( \eta, \phi \right)}{k_{\eta, \phi}\omega _{\eta}^{\prime}\omega _{\phi}}\in U^0$, then the assumption implies
		$
		\sum_{\left( \eta, \phi \right)}{k_{\eta, \phi}\varrho ^{\lambda, \mu}\left( \omega _{\eta}^{\prime}\omega _{\phi} \right)}=\sum_{\left( \zeta, \psi \right)}{k_{\zeta, \psi}\varrho ^{\sigma \left( \lambda \right), \mu}\left( \omega _{\zeta}^{\prime}\omega _{\psi} \right)}
		$,
		hence we have an equation for characters:
		\begin{equation*}
			\sum_{(\eta, \phi)} k_{\eta, \phi} \kappa_{\eta, \phi} = \sum_{(\zeta, \psi)} k_{\zeta, \psi} \kappa_{\zeta, \psi}^i.
		\end{equation*}
		For each $k_{\eta,\phi}\ne 0$, there exists a pair $(\zeta,\psi)\in Q \times Q$, such that $\kappa _{\eta, \phi}=\kappa _{\zeta, \psi}^{i}$ and $k_{\zeta,\psi}=k_{\eta,\phi}$. Then
		\begin{equation*}
			\kappa_{\eta, \phi}(0, \varpi_j) = \varrho^{0, \varpi_j}(\omega_{\eta}^{\prime}\omega_{\phi})
			=q^{(\eta + \phi, \varpi_j)}
			=\\{\kappa^i}_{\zeta, \psi}(0, \varpi_j) = \varrho^{0, \varpi_j}(\omega_{\zeta}^{\prime} \omega_{\psi})
			= q^{(\zeta + \psi, \varpi_j)},
		\end{equation*}
		yields $\eta+\phi=\zeta+\psi$, and
		$
		\varrho ^{\varpi _i}\left( \omega_{\eta}^{\prime}\omega _{\phi} \right) =\varrho ^{\sigma _i\left( \varpi _i \right)}\left( \omega _{\zeta}^{\prime}\omega _{\psi} \right) =\varrho ^{\varpi _i-\alpha _i}\left( \omega _{\zeta}^{\prime}\omega _{\psi} \right)
		$.
		Substituting $\phi=\zeta+\psi-\eta$, one has
		$
		\varrho ^{\varpi _i}\left( \omega _{\eta -\zeta}^{\prime}\omega _{\zeta -\eta} \right) \varrho ^{\alpha _i}\left( \omega _{\zeta}^{\prime}\omega _{-\zeta} \right) =\varrho ^{-\alpha _i}\left( \omega _{\psi +\zeta} \right)
		$,
		that is,
		\begin{equation*}
			q^{2(\zeta - \eta)^T DA\varpi_i-2\zeta^T DA\alpha_i}
			=
			q^{(\psi + \zeta)^T DA(-\alpha_i)} \left( \prod_{j=1}^n q_{ij}^{(\psi + \zeta)_i} \right).
		\end{equation*}
		For $i=1,...,n$, the algebraic independence of the parameters implies
		\begin{equation*}
			\begin{cases}
				2(\zeta - \eta)^T DA \varpi_i + (\psi - \zeta)^T DA \alpha_i = 0, \\
				(\psi + \zeta)_j = 0, \quad j \ne i.
			\end{cases}
		\end{equation*}
		So we have $\psi+\zeta=0$, which forces
		\begin{equation*}
			u = \sum_{(\eta, -\eta)} k_{\eta, -\eta} \, \omega_{\eta}^{\prime} \omega_{-\eta} \in U_\flat^0.
		\end{equation*}
		Finally, by Lemma \ref{weyl group act}, we get
		$
		\varrho ^{\lambda, \mu}\left( \sigma ^{-1}\left( u \right) \right) =\varrho ^{\sigma \left( \lambda \right), \mu}\left( u \right) =\varrho ^{\lambda, \mu}\left( u \right)
		$
		for all $\lambda,\mu \in \Lambda$ and $\sigma \in W$.
		It implies $\sigma^{-1}(u)=u$ for all $\sigma \in W$. So $u\in (U_\flat^0)^W$.
	\end{proof}
	
	\begin{theorem}\label{contain half1}
		$\varrho ^{\lambda +\rho, \mu}\left( \xi \left( z \right) \right) =\varrho ^{\sigma \left( \lambda +\rho \right), \mu}\left( \xi \left( z \right) \right)$,
		for all $z\in Z(U)$, $\sigma \in W$, which implies $Im(\xi)\subseteq (U_\flat^0)^W$.
	\end{theorem}
	
	\begin{proof}
		Let $z\in Z(U)$ and $\mu \in \Lambda$, take a $\lambda \in \Lambda$ such that $(\lambda,{\alpha _i}^{\vee})\geqslant 0$ for some fixed $i$. Let $v_{\lambda,\mu}$ be the highest weight vector of the Verma module $M(\varrho ^{\lambda, \mu}
		)$. Then
		\begin{equation*}
			z.v_{\lambda,\mu}=\varrho^{\lambda,\mu}(\pi(z))v_{\lambda,\mu}=\varrho^{\lambda+\rho,\mu}(\xi(z))v_{\lambda,\mu}.
		\end{equation*}
		On the other hand, according to \cite[Lem.~33]{PHR10}, we have
		\begin{equation*}
			e_i{f_i}^{\left( \lambda, \alpha _{i}^{\vee} \right) +1}.v_{\lambda, \mu}
			=\left[ (\left( \lambda, \alpha _{i}^{\vee} \right) +1) \right]_i
			f_{i}^{\left( \lambda, \alpha _{i}^{\vee} \right)}
			\frac{q_{i}^{-\left( \lambda, \alpha _{i}^{\vee} \right)}\omega_i - q_{i}^{\left( \lambda, \alpha _{i}^{\vee} \right)}\omega _{i}^{\prime}}{q_i-q_{i}^{-1}}.v_{\lambda, \mu} = 0,
		\end{equation*}
		since
		$q_{i}^{-\left( \lambda, \alpha _{i}^{\vee} \right)}\varrho ^{\lambda, \mu}\left( \omega _i \right)=q_{i}^{\left( \lambda, \alpha _{i}^{\vee} \right)}\varrho ^{\lambda, \mu}\left( \omega _{i}^{\prime} \right)$.
		It follows that $e_jf_{i}^{\left( \lambda, \alpha _{i}^{\vee} \right) +1}.v_{\lambda, \mu}=0$, $j=1,\cdots,n$, and
		\begin{equation*}
			\begin{split}
				z f_{i}^{\left( \lambda, \alpha _{i}^{\vee} \right) +1}.v_{\lambda, \mu}
				&= \pi(z) f_{i}^{\left( \lambda, \alpha _{i}^{\vee} \right) +1}.v_{\lambda, \mu}=\varrho ^{\sigma _i\left( \lambda +\rho \right) -\rho, \mu} \left( \pi \left( z \right) \right)  f_{i}^{\left( \lambda, \alpha _{i}^{\vee} \right) +1}.v_{\lambda, \mu}\\
				&= \varrho ^{\sigma _i\left( \lambda +\rho \right), \mu} \left( \xi \left( z \right) \right)  f_{i}^{\left( \lambda, \alpha _{i}^{\vee} \right) +1}.v_{\lambda, \mu}.
			\end{split}
		\end{equation*}
		Hence, $z$ acts on $M(\varrho ^{\lambda, \mu})$ by scalar
		$\varrho ^{\sigma _i\left( \lambda +\rho \right), \mu}\left( \xi \left( z \right) \right)$,
		so we have
		$\varrho ^{\lambda +\rho, \mu}\left( \xi \left( z \right) \right) =\varrho ^{\sigma _i\left( \lambda +\rho \right), \mu}\left( \xi \left( z \right) \right)$.
		
		Moreover, this result holds for any $\lambda \in \Lambda$. This is because if $\left( \lambda, \alpha _{i}^{\vee} \right)
		=-1$, then $\lambda+\rho=\sigma_i(\lambda^{\prime}+\rho)$ such that this equation holds. If $
		\left( \lambda, \alpha _{i}^{\vee} \right) <-1$, let $\lambda^{\prime}=\sigma_i(\lambda+\rho)-\rho$, then
		$
		\left( \lambda ^{\prime},\alpha _{i}^{\vee} \right) \geqslant 0
		$ such that the equation holds for $\lambda^\prime$. Replace $\lambda^\prime$ with $
		\sigma _i\left( \lambda +\rho \right) -\rho
		$ into the result such that the equation holds for $\lambda$ in this case. Finally, since the equation holds for each $\sigma_i$, so it holds for all $\sigma \in W$. It implies $Im(\xi)\subseteq (U_\flat^0)^W$ by Lemma 6.
	\end{proof}
	
	\section{Central elements and the Harish-Chandra theorem}
	
	\subsection{Rosso form realization and central elements} In what follows, we construct central elements by realizing the quantum trace on weight modules using the Rosso form.
	
	Since the Rosso form is nondegenerate and $\operatorname{ad}$-invariant, there is an injective morphism of $U$-module $\beta: U\rightarrow U^*,u\mapsto \langle\, u\mid-\,\rangle$, where the $U$-module structure on $U^*$ is defined by
	$
	(x.f)(v) = f\big( \operatorname{ad}_l(S(x)) v \big).
	$
	for all $x,v\in U,\;f\in U^*$.
	
	\begin{definition}
		For $\lambda \in \Lambda^+$, define the quantum trace $t_\lambda \in U^*$ by
		\begin{equation*}
			t_\lambda(v) := \operatorname{tr}_{L(\lambda)}(v \circ \Theta),
		\end{equation*}
		where $\Theta \in \operatorname{End}_\mathbb{K} (L(\lambda))$ is a linear map defined by
		$m\mapsto q^{-2{(\rho,\mu)}} m$, $\forall\; m\in L(\lambda)_\mu, \; \mu \in \Lambda$.
	\end{definition}
	
	Let $d = \dim L(\lambda)$, and choose a basis $\{ m_i \}_{i=1}^d$ of $L(\lambda)$ with the corresponding dual basis $\{ f_i \}_{i=1}^d \subset L(\lambda)^*$. Then, for any $v \in U$,
	\begin{equation*}
		v \cdot \Theta(m_i) = \sum_{j=1}^d f_j(v \cdot \Theta(m_i)) m_j = \sum_{j=1}^d C_{f_j, \Theta(m_i)}(v) m_j,
	\end{equation*}
	where $C_{f,m}\in U^*$ denotes the matrix coefficient satisfying $C_{f,m}(v)=f(v.m)$, for all $v\in U$. The quantum trace can then be expressed as
	\begin{equation*}
		t_\lambda(v) = \operatorname{tr}_{L(\lambda)}(v \circ \Theta) = \sum_{i=1}^d C_{f_i, \Theta(m_i)}(v)
	\end{equation*}
	and the matrix coefficient can be realized by the Rosso form.
	
	\begin{proposition}
		If the weight set of a finite-dimensional weight module $M$ of $U$ satisfies $\operatorname{wt}(M)\subseteq Q$, then for each $m \in M$, $f \in M^*$, there exists a unique element $u \in U$ such that $C_{f,m}(v)=\langle\, u \mid v\, \rangle$, $\forall \; v \in U$.
	\end{proposition}
	
	\begin{proof}
		It suffices to verify the case when $m \in M_\lambda$, since $C_{f,m}$ is linear in $m \in M$.
		Suppose $v$ is a monomial of the form
		$v = (y \, {\omega'_\mu}^{-1}) \omega'_\eta \omega_\phi x$, where
		$y \in U^-_{-\mu},\; x \in U^+_\nu$.
		Then for any $f \in M^*$, we have
		\begin{equation*}
			\begin{split}
				C_{f,m}(v)
				&= C_{f,m}\big((y {\omega'_\mu}^{-1})\omega'_\eta \omega_\phi x\big) \\
				&= \langle \omega'_\eta,  \omega_{-(\lambda+\nu)} \rangle \langle \omega'_{\lambda+\nu},  \omega_\phi \rangle \cdot f\big((y {\omega'_\mu}^{-1}) x \cdot m\big).
			\end{split}
		\end{equation*}
		Noting the bilinearity of
		$
		\varPsi: U^-_{-\mu} \times U^+_\nu \to \mathbb{K},\; (y,x) \mapsto f\big((y {\omega'_\mu}^{-1}) x \cdot m\big)
		$
		, we claim that for a given $\varPsi$ and a pair $(\eta,\phi)\in Q\times Q$, there exists an element $u\in U_{-\mu}^{-}U^0U_{v}^{+}$ such that
		\begin{equation*}
			\langle\, u\mid (y \omega_{\mu}^{\prime -1}) \, \omega_{\zeta}^{\prime} \omega_{\psi} x \,\rangle
			= \langle \omega_{\zeta}^{\prime}, \omega_{\phi} \rangle \, \langle \omega_{\eta}^{\prime}, \omega_{\psi} \rangle \, \varPsi(y, x)
		\end{equation*}
		for any $x\in U_v^+,y\in U_{-\mu}^-$ and $(\zeta,\psi)\in Q \times Q$. One can easily check that
		\begin{equation*}
			u= \sum_{i,j} q^{-2(\rho, \nu)} \varPsi \left( v_{j}^{\mu}, u_{i}^{\nu} \right) \left( v_{i}^{\nu} \omega_{\nu}^{\prime -1} \right) \omega_{\eta}^{\prime} \omega_{\phi} u_{j}^{\mu},
		\end{equation*}
		where $\left\{ u_{1}^{\mu},...,u_{d_{\mu}}^{\mu} \right\}$ is a basis of $U_{\mu}^{+}$ with its dual basis $\left\{ v_{1}^{\mu},...,v_{d_{\mu}}^{\mu} \right\} $ in $U_{-\mu}^{-}$ with respect to $\langle -,-\rangle$.
		Therefore, there exists a unique $u_{\nu\mu} \in U$ such that
		$
		C_{f,m}(v) = \langle\, u_{\nu\mu} \mid v \,\rangle, \,\forall\; v \in U^-_{-\mu} U^0 U^+_\nu.
		$
		
		Now consider a general element $v \in U$ and write it as
		$v = \sum_{(\mu,\nu)} v_{\mu\nu}$, where $v_{\mu\nu} \in U^-_{-\mu} U^0 U^+_\nu$.
		Since $M$ is finite-dimensional, there exists a finite subset $\Omega \subset Q \times Q$ such that
		$C_{f,m}(v) = C_{f,m} \left( \sum_{(\mu,\nu) \in \Omega} v_{\mu\nu} \right)$
		for all $v \in U$.
		Define $u := \sum_{(\mu,\nu) \in \Omega} u_{\nu\mu}$,
		then we have
		\begin{equation*}
			\langle\, u \mid v \,\rangle
			= \sum_{(\mu,\nu) \in \Omega} \langle\, u_{\nu\mu} \mid v\, \rangle
			= \sum_{(\mu,\nu) \in \Omega} \langle\, u_{\nu\mu} \mid v_{\mu\nu}\, \rangle
			= C_{f,m}(v).
		\end{equation*}
		
		This proves the existence of $u$ and provides an explicit construction.
	\end{proof}
	
	\begin{theorem}\label{form of z}
		For each $\lambda \in \Lambda^+ \bigcap Q$, there exists a unique  element $z_\lambda$ such that $\beta(z_\lambda) = t_\lambda$.
		The explicit expression of $z_\lambda$ is
		\begin{equation*}
			z_\lambda = \sum_{\tau \leqslant \lambda}
			\sum_{\mu \in Q^+}
			\sum_{i,j} q^{-2(\rho,\tau+\mu)}
			\langle {\omega'_\mu}, \omega_{\tau+\mu} \rangle
			tr(v^\mu_j u^\mu_i \circ P_\tau) \hspace{0.1cm}
			v^\mu_i \omega'_{\tau} \omega^{-1}_{\tau+\mu} u^\mu_j.
		\end{equation*}
		where $\{u^\mu_j\}_{j=1}^{d_\mu}$ is a basis of $U^+_\mu$ and $\{v^\mu_i\}_{i=1}^{d_\mu}$ is the dual basis of $U^-_{-\mu}$ with respect to the restriction of $\langle-,- \rangle$ to $U^-_{-\mu} \times U^+_{\mu}$, and $P_\tau$ is the projector from $L(\lambda)$ to $L(\lambda)_\tau$.
	\end{theorem}
	
	\begin{remark}
		The explicit expression of central elements $z_\lambda$ in the one-parameter quantum groups $U_q(\mathfrak{g})$
		(see, e.g. \cite{D23,J96}) corresponds to the specialization of $U_{q,G}$ where all $q_{ij}=1,\; \omega_i = K_i,\; \omega'_i= K_i^{-1}$.
	\end{remark}
	
	We proceed to show that these elements $\{z_\lambda\}_{\lambda \in \Lambda^+ \cap Q}$ lie in the centre $Z(U)$.
	
	\begin{lemma} \label{in ZU if and only if}
		Let $z\in U$, then $z\in Z(U)$ if and only if $ad_l(x)z=\varepsilon (x)z$, for all $x\in U$.
	\end{lemma}
	
	\begin{proof}
		Suppose $z\in Z(U)$. Then for all $x\in U$, we have
		\begin{equation*}
			\operatorname{ad}_l(x)z = \sum_{(x)} x_{(1)} z S\bigl(x_{(2)}\bigr)
			= z \sum_{(x)} x_{(1)} S\bigl(x_{(2)}\bigr)
			= \varepsilon(x) z.
		\end{equation*}
		Conversely, if $\operatorname{ad}_l(x)z=\varepsilon(x)z$ holds for all $x\in U$, then
		$
		\omega_i z \omega_i^{-1} = \operatorname{ad}_l(\omega_i) z = \varepsilon(\omega_i) z = z.
		$
		For each generator $e_i$ and $f_i$ of $U$, we have
		\begin{equation*}
			\begin{split}
				e_i z - z e_i
				&= e_i z + \omega_i z S(e_i)
				= \operatorname{ad}_l(e_i) z  = \varepsilon(e_i) z = 0,\\
				\left( f_i z - z f_i \right) {\omega_i'}^{-1}
				&= z S(f_i) + f_i z S(\omega_i')
				= \operatorname{ad}_l(f_i) z
				= \varepsilon(f_i) z =0.
			\end{split}
		\end{equation*}
		
		This completes the proof.
	\end{proof}
	
	\begin{proposition}
		The following statements hold for the weight modules $L(\lambda)$ with $\lambda \in \Lambda^+$.
		\begin{enumerate}
			\item For all $u\in U$, we have $\Theta(u.m) = S^2(u).\Theta(m)$, $\forall\; m\in L(\lambda)$, or $\Theta \circ u=S^2(u) \circ \Theta$ for short.
			\item For all $x\in U$, we have $x \cdot t_\lambda = \varepsilon(x) \; t_\lambda$.
			\item If $\lambda \in \Lambda^+ \cap Q$, then $z_\lambda = \beta^{-1}(t_\lambda) \in Z(U)$.
		\end{enumerate}
	\end{proposition}
	
	\begin{proof}
		(1) It is sufficient to verify the statement on the generators $e_i$ and $f_i$. Notice that
		\begin{equation*}
			\langle \omega'_i, \omega_i \rangle
			= a_{ii} = q^{2d_i}
			= q^{2(\varpi_i, \alpha_i)} = q^{2(\rho, \alpha_i)}.
		\end{equation*}
		Now, for each $m \in L(\lambda)_\mu$, we compute
		\begin{equation*}
			\begin{split}
				S^2(e_i) \cdot \Theta(m)
				&= q^{-2(\rho, \mu)} \cdot S^2(e_i) \cdot m
				= q^{-2(\rho, \mu)} \cdot \langle \omega'_i, \omega_i \rangle^{-1} e_i \cdot m \\
				&= q^{-2(\rho, \mu + \alpha_i)} e_i \cdot m
				= \Theta(e_i \cdot m), \\
				S^2(f_i) \cdot \Theta(m)
				&= S^2(f_i) \cdot q^{-2(\rho, \mu)} m
				= \langle \omega'_i, \omega_i \rangle \cdot q^{-2(\rho, \mu)} f_i \cdot m \\
				&= q^{-2(\rho, \mu - \alpha_i)} f_i \cdot m
				= \Theta(f_i \cdot m).
			\end{split}
		\end{equation*}
		
		(2) For any $x, u \in U$, we compute
		\begin{equation*}
			\begin{split}
				(&S^{-1}(x) \cdot t_\lambda)(u)
				= t_\lambda(\operatorname{ad}_l(x) u) \hspace{0.1cm}
				= \operatorname{tr}_{L(\lambda)}\big(\sum_{(x)} x_{(1)} u S(x_{(2)}) \circ \Theta \big) \\
				&= \operatorname{tr}_{L(\lambda)}\big(u \cdot \sum_{(x)} S(x_{(2)}) \Theta x_{(1)} \big)
				= \operatorname{tr}_{L(\lambda)}\big(u \cdot \sum_{(x)} S(x_{(2)}) S^2(x_{(1)}) \circ \Theta \big) \\
				&= \operatorname{tr}_{L(\lambda)}\big(u \cdot S(\sum_{(x)} S(x_{(1)}) x_{(2)}) \circ \Theta \big)
				= \varepsilon(x) \operatorname{tr}_{L(\lambda)}(u \circ \Theta) = \varepsilon(x) t_\lambda(u).
			\end{split}
		\end{equation*}
		Replacing $x$ by $S(x)$, we get
		$x \cdot t_\lambda = \varepsilon(x) t_\lambda$ for all $x \in U$.

		(3) We compute that
		\begin{equation*}
			\begin{split}
				x \cdot t_\lambda
				&= x \cdot \beta(z_\lambda) = (x \cdot \beta)(z_\lambda) = \beta(\operatorname{ad}_l(x) z_\lambda), \\
				\varepsilon(x) t_\lambda
				&= \varepsilon(x) \beta(z_\lambda) = \beta(\varepsilon(x) z_\lambda).
			\end{split}
		\end{equation*}
		By the injectivity of $\beta$, it follows that
		$\operatorname{ad}_l(x)(z_\lambda) = \varepsilon(x) z_\lambda$ for all $x \in U$, which implies $z_\lambda \in Z(U)$ by Lemma \ref{in ZU if and only if}.
	\end{proof}
	
	\subsection{The Harish-Chandra theorem}\label{HC}
	
	Set $\Psi = \Lambda^+ \cap Q$. For each $\eta \in Q$, there exists a unique $\sigma \in W$ such that $\sigma(\eta) \in \Psi$. Therefore, the elements
	\begin{equation*}
		\operatorname{av}(\lambda) = \frac{1}{|W|} \sum_{\sigma \in W} \omega'_{\sigma(\lambda)} \omega_{-\sigma(\lambda)} \quad (\lambda \in \Psi)
	\end{equation*}
	form a basis of the $W$-invariant subalgebra $(U^0_\flat)^W$.
	
	\begin{theorem}\label{iso}
		The Harish-Chandra homomorphism $\xi:$ $Z(U)\rightarrow (U_\flat^0)^W$ is an algebra isomorphism.
	\end{theorem}
	
	\begin{proof}
		By Theorem \ref{form of z}, for each $\lambda \in \Lambda^+ \cap Q$, the central element $z_\lambda$ satisfies
		\begin{equation*}
			z_\lambda^0 = \sum_{\mu \leqslant \lambda} q^{-2(\rho, \mu)} \dim(L(\lambda)_\mu) \, \omega'_\mu \omega_{-\mu}.
		\end{equation*}
		Applying the definition of the map $\xi$ in Section \ref{sec3.1}, we obtain
		\begin{equation*}
			\begin{split}
				\xi(z_\lambda) = \gamma^{-\rho}(z_\lambda^0)
				&= \sum_{\mu \leqslant \lambda} q^{-2(\rho, \mu)} \dim(L(\lambda)_\mu) \, \varrho^{-\rho}(\omega'_\mu \omega_{-\mu}) \, \omega'_\mu \omega_{-\mu} \\
				&= \sum_{\mu \leqslant \lambda} \dim(L(\lambda)_\mu) \, \omega'_\mu \omega_{-\mu} \; \in (U^0_\flat)^W.
			\end{split}
		\end{equation*}
		
		It remains to prove that each $\operatorname{av}(\lambda)$ lies in $\operatorname{Im}(\xi)$. We proceed by induction on the height of $\lambda$.
		
		If $\lambda = 0$, then $\operatorname{av}(0) = 1 = \xi(z_0) \in \operatorname{Im}(\xi)$. Suppose $\lambda > 0$. The properties
		$
		\dim(L(\lambda)_\mu) = \dim(L(\lambda)_{\sigma(\mu)})
		$
		and $\dim(L(\lambda)_\lambda) = 1$ imply that
		\begin{equation*}
			\xi(z_\lambda) = \sum_{\mu \leqslant \lambda} \dim(L(\lambda)_\mu) \, \omega'_\mu \omega_{-\mu}
			= |W| \cdot \operatorname{av}(\lambda) + |W| \sum_{\mu < \lambda,\; \mu \in \Psi} \dim(L(\lambda)_\mu) \operatorname{av}(\mu).
		\end{equation*}
		By the induction hypothesis, each $\operatorname{av}(\mu)$ for $\mu < \lambda$ lies in $\operatorname{Im}(\xi)$, then
		\begin{equation*}
			\operatorname{av}(\lambda) = \frac{1}{|W|} \xi(z_\lambda) - \sum_{\mu < \lambda,\; \mu \in \Psi} \dim(L(\lambda)_\mu) \operatorname{av}(\mu) \in \operatorname{Im}(\xi).
		\end{equation*}
		Now we conclude that
		\begin{equation*}
			\begin{split}
				(U^0_\flat)^W &= \operatorname{span}_\mathbb{K} \{ \operatorname{av}(\lambda) \mid \lambda \in \Psi \}
				\subseteq \operatorname{Im}(\xi),
			\end{split}
		\end{equation*}
		which completes the proof together with Theorem \ref{contain half1}.
	\end{proof}
	
	\begin{corollary} \label{conj}
		The centre $Z(U)$ admits the basis $\{z_\lambda\;\big|\;\lambda\in\Psi\}$ over $\mathbb{K}$, thereby supporting the Batra-Yamane conjecture \cite[Conj.~3.13]{BY20}.
	\end{corollary}
	
	\begin{proof}
		The previous proof has shown that
		$(U^0_\flat)^W = \operatorname{span}_\mathbb{K} \{ \operatorname{av}(\lambda) \mid \lambda \in \Psi \} =  \operatorname{span}_\mathbb{K} \{ \xi(z_\lambda) \mid \lambda \in \Psi \}$. Moreover, the set $\{\xi(z_\lambda)\mid \lambda\in\Psi\}$ is a $\mathbb{K}$-linearly independent, otherwise $\exists\; \lambda^1,\cdots,\lambda^s \in \Psi$, such that
		\begin{equation*}
			0 = \sum_{i=1}^{s} a_i \xi(z_{\lambda^i}) = |W| \sum_{i=1}^{s} a_i \left( \operatorname{av}(\lambda^i) + \sum_{\mu < \lambda^i,\; \mu \in \Psi} \operatorname{dim}(L(\lambda^i)_\mu) \operatorname{av}(\mu) \right),
		\end{equation*}
		and all $a_i \neq 0$, we thus arrive at a contradiction: there exists a weight $\lambda^{i_0}$ such that $\lambda^i \ngtr \lambda^{i_0}$ for all other $i$, which necessarily implies $a_{i^0} \operatorname{av}(\lambda^{i^0}) = 0$ in the summand.
		Hence the set $\{\,\xi(z_\lambda)\mid\lambda\in\Psi\,\}$ forms a basis of $(U^0_\flat)^W$. By applying the Harish-Chandra isomorphism, we conclude that the centre $Z(U)$ admits the basis $\{z_\lambda\mid\lambda\in\Psi\}$.
		
		Following the notation established in \cite{BY20} and under the assumptions of the present work, we set $\pi(i)=\alpha_i$, $\mathfrak{A}=Q$, and the map $\chi(\alpha_i, \alpha_j) = \langle w_j' , w_i \rangle$ for all $i,j$. Then the set $\mathfrak{Z}_1 (\chi,\pi)$ is just the centre $Z(U)$, and $\mathfrak{B}^{\chi,\pi}_1 = (U^0_\flat)^W$ --- an explicit realization of the abstract subalgebra in \cite[Thm.~10.4]{BY18}. The set
		\begin{equation*}
			\operatorname{Fin}^\chi_1:= \{(\lambda,\mu) \in \mathfrak{A} \times \mathfrak{A} \mid \dim \mathcal{L}(\Lambda^\chi_{\lambda,\mu;1}) < \infty\}
		\end{equation*}	
		where the map
		$\Lambda^\chi_{\lambda,\mu;1}(w'_{\mu'} w_{\lambda'}) := \chi(\lambda,\mu') \chi(\lambda',\mu)=
		\langle w'_{\mu},w_{\lambda'}\rangle  \langle w'_{\mu'},w_\lambda \rangle
		$ (also denoted by $\varrho^{(\lambda,\mu)}$).
		Our $\varrho^\lambda = \Lambda^\chi_{-\lambda,\lambda;1} = \varrho^{(-\lambda,\lambda)}$, then $(-\lambda,\lambda) \in \operatorname{Fin}^\chi_1$ if and only if $\lambda\in \Lambda^+ \cap \mathfrak{A} = \Psi$. On the other hand, there is no finite-dimensional weight module with the highest weight $\varrho^{(\lambda,\mu)}$ such that $\lambda+\mu \neq 0$, this fact is guaranteed by \cite[Prop.36]{PHR10} through quick verification. Above all, we have $\Psi=\operatorname{Fin}^\chi_1$ and confirm that the Batra-Yamane conjecture $\mathfrak{Z}_1 (\chi,\pi)=\operatorname{span}_{\mathbb{K}}\left\{ z_{\lambda} \middle| \lambda \in \operatorname{Fin}^\chi_1 \right\}$ holds for $Z(U_{q,G}(\mathfrak{g}))$.
	\end{proof}
	
	\section{The Centre of $U_{q,G}(\mathfrak{g})$}
	
	Now we determine generators of the subalgebra $(U_{\flat}^0)^W$.
	There is a natural monomorphism
	$\theta : (U^0_\flat)^W \rightarrow \mathbb{K}[\mathfrak{h}^*]^W$
	defined by
	\begin{equation*}
		\theta\left( \sum_{\sigma \in W} \omega^\prime_{\sigma(\lambda)}\omega_{-\sigma(\lambda)} \right) = \sum_{\sigma \in W} e^{\sigma(\lambda)}, \quad \forall \;\lambda \in  \Psi.
	\end{equation*}
	Meanwhile, the homomorphism $\operatorname{Ch}: R(\mathfrak{g}) \rightarrow \mathbb{Z}[\Lambda]$ induces an isomorphism $R(\mathfrak{g}) \cong \mathbb{Z}[\Lambda]^W$ (see, e.g., \cite{BP02}) and $R(\mathfrak{g})\otimes_{\mathbb{Z}}  \mathbb{K} \cong \mathbb{K}[\mathfrak{h}^*]^W$.
	Since the element $\textsc{z}_\lambda := \textnormal{Ch}(L(\lambda))$ satisfies $\textsc{z}_\lambda = \theta \xi(z_\lambda) $ when $\lambda \in \Psi$, we have
	$\mathrm{Im}(\theta) = \mathrm{span}_\mathbb{K} \{\ \textsc{z}_\lambda \mid \lambda \in \Psi \}$.
	For simplicity of notation, we write
	$\textsc{z}_i := \textsc{z}_{\varpi_i}$, $i=1,\cdots,n$.
	
	\begin{lemma}\label{alg ind}
		The elements $\textsc{z}_1,\; \textsc{z}_2,\;\cdots,\;\textsc{z}_{n}$ are algebraically independent over $\mathbb{K}$.
	\end{lemma}
	\begin{proof}
		If there exists a polynomial $f$ in $\mathbb{K}[v_1,\cdots,v_n]$
		such that $f(\textsc{z}_1, \textsc{z}_2, \dots, \textsc{z}_n) = 0$, writing
		\begin{equation*}
			f = c_{k_1,\dots,k_n} v_1^{k_1} v_2^{k_2} \dots v_n^{k_n} + \sum_{\left( a_1,...,a_n \right) <\left( k_1,...,k_n \right)}{c_{a_{1,...,}a_n}v_{1}^{a_1}v_{2}^{a_2}...v_{n}^{a_n}},
		\end{equation*}
		where ``$<$'' denotes the lexicographical order, and $v_1^{k_1} v_2^{k_2} \dots v_n^{k_n}$ is the maximal one with $c_{k_1, \dots, k_n} \neq 0$.
		Through the isomorphism $\operatorname{Ch}:R(\mathfrak{g})\otimes_{\mathbb{Z}}  \mathbb{K} \cong \mathbb{K}[\mathfrak{h}^*]^W$, the equation $f(\textsc{z}_1, \textsc{z}_2, \dots, \textsc{z}_n)=0$ corresponds to
		\begin{equation*}
			c_{k_1,\dots,k_n} \left[ \bigotimes_{i=1}^n L(\varpi_i)^{\otimes k_i}\right] +  \sum_{\left( a_1,...,a_n \right) <\left( k_1,...,k_n \right)} c_{a_1,...,a_n}  \left[ \bigotimes_{i=1}^n L(\varpi_i)^{\otimes a_i}\right] =0.
		\end{equation*}
		If we rewrite the equation above as a linear combination of the basis elements in $\{[L(\lambda)]\mid \lambda \in \Lambda^+\}$,
		then $[L(k_1\varpi _1+...+k_n\varpi _n)]$ only appears in $\left[ \bigotimes_{i=1}^n L(\varpi_i)^{\otimes k_i}\right]$, then $c_{k_1, \dots, k_n}$ has to be $0$, contradiction.
	\end{proof}
	
	\subsection{Minimal generating system $\Psi_{\min}$}
	An element $\lambda \in \Psi \setminus \{0\}$ is said to be indecomposable if  $\nexists\; \lambda', \lambda'' \in \Psi \setminus {0}$, such that $\lambda = \lambda' + \lambda''$. Let $\Psi_{\min}$ be the set containing all indecomposable weights in $\Psi \setminus \{0\}$, which is the minimal generating system of $\Psi$. We adopt an approach analogous to that in \cite{LXZ16} and adapt it to our multi-parameter setting. We will show that the following lemmas, together with our main result (Theorem \ref{Last theorem}), hold in this setting, where the constructions of $\Psi$ and $\Psi_{min}$ are considerably more refined.
	\begin{lemma}
		The set $\Psi_{\min}$ is finite.
	\end{lemma}
	
	\begin{proof}
		According to the table in Appendix A, for each $1\leqslant k \leqslant n$, there exist elements $c \in \mathbb{Z}_+$ such that $c \varpi_k \in \Psi\setminus\{0\}$. The smallest $c$ is denoted by $r_k$ (e.g., $r_k=\textstyle\frac{n+1}{(n+1,k)}$ for type $A_n$), and the weight $\textbf{s}(k) := r_k \varpi_k \in \Psi_{\min}$ is called a single weight.
		Hence if $\lambda=\sum_{i=1}^{n} \lambda_i \varpi_i \in \Psi_{\min}$, then each $\lambda_{i} \leqslant r_i$ (otherwise $\exists\, i_0$, such that $\lambda - \textbf{s}(i_0) \in \Psi\setminus\{0\}$, contradiction). The finiteness is clear.
	\end{proof}
	
	Except for type $A_n$, one can directly check that the given set forms a minimal generating system for the additive monoid $\Psi$ by straightforward enumeration and combinatorial arguments.
	
	\begin{lemma}\label{min generator}
		For $\operatorname{rank}(\mathfrak{g}) \geqslant 2$, the elements of $\Psi_{\min}$ are shown explicitly as follows
		\vspace{0.3cm}
		\begin{center}
			\begin{tabular}{|c|l|}
				\hline
				\textbf{Type} & \textbf{Minimal generating system $\Psi_{\min}$} \\
				\hline
				$A_n$ &
				$\begin{aligned}
					&\{\textbf{s}(k) := r_k \varpi_k\; \big| \;  r_k=\textstyle\frac{n+1}{(n+1,k)},\;  1\leqslant k \leqslant n\;\}_{\textnormal{(single)}} \\
					& \cup \{\textbf{e}(k) := d_k \varpi_1 + \varpi_k\;  \big| \; d_k=n+1-k,\; 2 \leqslant k \leqslant n\;\}_{\textnormal{(special)}} \\
					& \cup \{ \tau \mid  \text{not single and not special}\}_{=: T}\\
				\end{aligned}$ \\
				\hline
				$B_n$ &
				$\{\varpi_1, \dots, \varpi_{n-2}, \varpi_{n-1}, 2\varpi_n\}$ \\
				\hline
				$C_{2k+1}$ &
				$\begin{aligned}
					\{ & 2\varpi_1, \varpi_2, 2\varpi_3, \dots, \varpi_{2k}, 2\varpi_{2k+1}, \\
					& \varpi_u + \varpi_v \; (u,v \equiv 1 \pmod{2},\, u<v)  \}
				\end{aligned}$ \\
				\hline
				$C_{2k+2}$ &
				$\begin{aligned}
					\{ & 2\varpi_1, \varpi_2, 2\varpi_3, \dots, \varpi_{2k}, 2\varpi_{2k+1}, \varpi_{2k+2},\\
					& \varpi_u + \varpi_v \; (u,v \equiv 1 \pmod{2},\, u<v) \}
				\end{aligned}$ \\
				\hline
				$D_{2k+2}$ &
				$\begin{aligned}
					\{ & 2\varpi_1, \varpi_2, 2\varpi_3, \dots, 2\varpi_{2k-1}, \varpi_{2k}, 2\varpi_{2k+1}, 2\varpi_{2k+2}, \\
					& \varpi_u + \varpi_v \; (u,v \equiv 1 \pmod{2},\, u<v),\; \varpi_{2k+1} + \varpi_{2k+2} \}
				\end{aligned}$ \\
				\hline
				$D_{2k+3}$ &
				$\begin{aligned}
					\{ & 2\varpi_1, \varpi_2, 2\varpi_3, \dots, 2\varpi_{2k-1}, \varpi_{2k},
					2\varpi_{2k+1}, 4\varpi_{2k+2}, 4\varpi_{2k+3}, \\
					& \varpi_u + \varpi_v \; (u,v \equiv 1 \pmod{2},\, u<v\leqslant 2k+1), \\
					& \varpi_u + 2\varpi_{2k+2} \; (u \equiv 1 \pmod{2},\, u\leqslant 2k+1), \\
					& \varpi_u + 2\varpi_{2k+3} \; (u \equiv 1 \pmod{2},\, u\leqslant 2k+1), \\
					& 2\varpi_{2k+2} + 2\varpi_{2k+3}, \varpi_{2k+2} + 3\varpi_{2k+3}, 3\varpi_{2k+2} + \varpi_{2k+3} \}
				\end{aligned}$ \\
				\hline
				$E_6$ &
				$\begin{aligned}
					&\{3\varpi_1, \varpi_2, 3\varpi_3, \varpi_4, 3\varpi_5, 3\varpi_6, \\
					&\varpi_1 + \varpi_3, \varpi_1 + \varpi_6, \varpi_3 + \varpi_5, \varpi_5 + \varpi_6, \\
					&
					\varpi_1 + 2\varpi_5, 2\varpi_1 + \varpi_5, \varpi_3 + 2\varpi_6, 2\varpi_3 + \varpi_6 \}
				\end{aligned}$ \\
				\hline
				$E_7$ &
				$\begin{aligned}
					\{ & \varpi_1, 2\varpi_2, \varpi_3, \varpi_4, 2\varpi_5, \varpi_6, 2\varpi_7, \\
					& \varpi_2 + \varpi_5, \varpi_2 + \varpi_7, \varpi_5 + \varpi_7 \}
				\end{aligned}$ \\
				\hline
				$E_8$, $F_4$, $G_2$ &
				$\{\varpi_1,\; \varpi_2, \dots, \varpi_n\}$ \\
				\hline
			\end{tabular}
		\end{center} \vspace{0.3cm}
	\end{lemma}
	
	Here we give an explicit description of $\Psi_{\min}$ for type $A_n$.
	Let $\lambda=\sum_{i=1}^{n} \lambda_i \varpi_i \in \Lambda$. Set $|\lambda| := \sum_{k=1}^{n} k\lambda_k \in \mathbb{Z}_{+}$.
	Then we write
	\begin{equation*}
		\begin{split}
			\lambda &=\lambda_1\varpi_1-\textstyle\sum_{i=2}^n{\lambda_i(\alpha_{i-1}+2\alpha_{i-2}+...+(i-1)\alpha_1-i\varpi_1}) \\
			&=\left(\textstyle\sum^{n}_{i=1}i \lambda_i \right)\varpi_1-(\lambda_n\alpha_{n-1}+(2\lambda_n+\lambda_{n-1})\alpha_{n-2}+\cdots+((n-1)\lambda_n+...+\lambda_2)\alpha_1).
		\end{split}
	\end{equation*}
	Hence, $\lambda\in Q$ if and only if $(\sum_{i=1}^n{i \lambda_i})\varpi_1\in Q$, if and only if $(n+1) \big|\, |\lambda|$. This fact shows that
	\begin{equation*}
		\Psi = \left\{\lambda \in \Lambda^+ \mid\; (n+1)\big|\,|\lambda|\right\}.
	\end{equation*}
	Hence, single weights $\textbf{s}(k) = r_k \varpi_k \ (1\leqslant k \leqslant n)$ and special weights $\textbf{e}(k) = d_k \varpi_1 + \varpi_k \ (2 \leqslant k \leqslant n)$ are contained in $\Psi_{\min}$.
	We end this subsection by proving a property of the elements in $T$.

	\begin{lemma}\label{norm of tau}
		For each $\tau = \sum_{k=1}^{n} \tau_k \varpi_k \in T$, set $\|\tau\| =  -\tau_1 + \sum_{k \neq 1} \tau_k d_k$, then we have
		\begin{equation*}
			\tau + \|\tau\| \varpi_1 = \textstyle\sum_{k \neq 1} \tau_k \mathbf{e}(k),
		\end{equation*}	
		and $\|\tau\|$ is a positive integer such that $(n+1) \big| \,\|\tau\|$ .
	\end{lemma}
	
	\begin{proof}
		One can directly verify that
		\begin{equation*}
			\textstyle\sum_{k \neq 1} \tau_k \mathbf{e}(k)
			= \textstyle\sum_{k \neq 1} \tau_k (d_k \varpi_1 + \varpi_k)
			= \left( \textstyle\sum_{k \neq 1}  \tau_k d_k \varpi_1 \right) + \tau -\tau_1 \varpi_1 = \tau + \|\tau\|  \varpi_1.
		\end{equation*}
		Since $\sum_{k \neq 1} \tau_k \mathbf{e}(k)\in Q $ and $\tau \in Q$, we have $\|\tau\|\varpi_1 \in Q$, that is, $(n+1) \big| \|\tau\|$.
		
		Assume that $\|\tau\| \leqslant 0$, then $\tau_1 \geqslant \sum_{k \neq 1} \tau_k d_k$. Since $\tau\neq\tau_1 \varpi_1$ (otherwise $(n+1) \big| \tau_1$ and $\tau = \textstyle\frac{\tau_1}{n+1}\textbf{s}(1)$, contradiction),
		there exists $k_0 \neq 1$ such that $\tau_{k_0} \geqslant 1$. Then
		\begin{equation*}
			Q \ni \tau - \textbf{e}(k_0) = (\tau_1 - d_{k_0}) \varpi_1 + (\tau_{k_0} - 1) \varpi_{k_0} + \sum_{k \neq 1, k_0} \tau_k \varpi_k.
		\end{equation*}	
		However, $\tau_1 - d_{k_0} \geqslant \tau_1 - \sum_{k \neq 1} \tau_k d_k \geqslant 0$ implies that the right-hand side of the above expression belongs to $\Lambda^+$. This leads to $\tau-\textbf{e}(k_0) \in \Lambda^+ \cap Q = \Psi$, and then $\tau$ has to be $\textbf{e}(k_0)$ since $\tau$ is indecomposable, which contradicts to $\tau \in T$.
		Thus, we have $\|\tau\|>0$, the Lemma holds.
	\end{proof}
	
	\subsection{Presentation of $Z(U_{q,G})$ }
	
	Now we write $m = |\Psi_{\min}|$ and $\Psi_{\min} = \{\, \mu_1, \ldots, \mu_m\, \}$ with the order presented in the Table of Lemma \ref{min generator}.
	
	\begin{lemma} \label{minigenset}
		The set $\{\textsc{z}_{\mu_i} \}_{i=1}^m $ forms a minimal generating set of $ \operatorname{Im}(\theta) $.
		For each $ \mu_i = \sum_{j=1}^{n} \mu_{ij} \varpi_j \in \Psi_{\min} $, define another element $ x_{\mu_i} = \prod_{j=1}^{n} \textsc{z}_j^{\mu_{ij}} $.
		Then the set $ \{x_{\mu_i} \}_{i=1}^m $ provides another minimal generating set of $ \operatorname{Im}(\theta) $.
	\end{lemma}
	
	\begin{proof}
		Notice that $\operatorname{Im}(\theta) = \operatorname{span}_{\mathbb{K}}\left\{\, \textsc{z}_\lambda \mid \lambda \in \Psi \,\right\} $,
		if  $\lambda \in \Psi \setminus \Psi_{\min}$, then $\exists\, \mu \in \Psi_{\min},\; \lambda - \mu \in \Psi$, and
		\begin{equation*}
			\textsc{z}_{\lambda - \mu} \textsc{z}_\mu = \textsc{z}_\lambda + \sum_{\nu < \lambda,\; \nu \in \Lambda^+}
			[L(\lambda-\mu) \otimes L(\mu):L(\nu)]\;
			\textsc{z}_\nu,
		\end{equation*}
		which leads to $\textsc{z}_\lambda \in \langle \,\textsc{z}_\mu,\;\textsc{z}_{\lambda-\mu},\; \textsc{z}_\nu  \mid\; 0 < \nu < \lambda ,\; \nu \in \Psi \,\rangle $.
		Hence by induction on the square length $(\lambda, \lambda)$ of  $\lambda$,
		one obtains $\textsc{z}_\lambda \in \langle \textsc{z}_{\mu_i}\rangle_{i=1}^m$. The first claim is clear.
		
		Similarly, if $\lambda = \sum_{i=1}^n \lambda_i \varpi_i \in \Psi$, then
		\begin{equation} \label{decomp}
			\prod_{i=1}^n \textsc{z}_{i}^{\lambda_i} = \textsc{z}_\lambda + \sum_{\gamma < \lambda, \; \gamma \in \Lambda^+} d^\gamma_\nu(\lambda_1, \ldots, \lambda_n) \; \textsc{z}_\gamma.
		\end{equation}
		where $ d^{\gamma}_{\lambda}(\lambda_1, \cdots, \lambda_n) := [L(\varpi_1)^{\otimes \lambda_1} \otimes \cdots \otimes L(\varpi_n)^{\otimes \lambda_n}:L(\gamma)]$.
		If $ d^\lambda_\nu(\lambda_1, \ldots, \lambda_n) \neq 0 $, then $ \lambda - \gamma \in Q $, which forces $ \gamma \in \Psi $, that is, $ \textsc{z}_\gamma \in \operatorname{Im}(\theta)$. Therefore \(\prod_{i=1}^n \textsc{z}_{i}^{\lambda_i} \in \operatorname{Im}(\theta)\), and in particular, each $ x_{\mu_i} = \prod_{j=1}^{n} \textsc{z}_j^{\mu_{ij}} \in \operatorname{Im}(\theta) = \langle \textsc{z}_{\mu_i} \rangle_{i=1}^m$.
		
		To prove each $\textsc{z}_{\mu_i}$ can be generated by $\{ x_{\mu_i} \mid \mu_i \in \Psi_{\min} \}$, we temporarily reorder $\Psi_{\min}$ such that if $\mu_i < \mu_j$ then $i>j$.
		Then Eq.(\ref{decomp}) leads to $x_{\mu_m}\in \textsc{z}_{\mu_m}+\mathbb{Z}$ and $x_{\mu_i}\in \textsc{z}_{\mu_i}+\mathbb{Z}[\textsc{z}_{\mu_{i+1}}, \cdots,\textsc{z}_{\mu_m}]\;(1\leqslant i<m)$.
		Then one has $\textsc{z}_{\mu_m}\in x_{\mu_m}+\mathbb{Z}$, and $z_{\mu_i}\in x_{\mu_i}+\mathbb{Z}[x_{\mu_{i+1}}, \cdots,x_{\mu_m}]$ by descending induction on $1\leqslant i <m$. That is, each $\textsc{z}_{\mu_j}$ belongs to $\langle x_{\mu_i} \rangle_{i=1}^m$, for $j=1,\cdots,m$.
		The second claim $\operatorname{Im}(\theta) = \langle x_{\mu_i}  \rangle_{i=1}^m$ is clear.
	\end{proof}
	In the following theorem, we present a polynomial algebra $S$ together with a relation ideal $I$ such that $S/I$ gives a presentation of $\langle x_{\mu_i} \rangle_{i=1}^m$, from which $Z(U)$ is obtained.
	
	\begin{theorem} \label{Last theorem}
		\begin{enumerate}
			\item The centre $Z(U)$ is isomorphic to a polynomial algebra $\mathbb{K}[t_1, \cdots, t_n]$ when $\mathfrak{g}$ is of type $B_n$, $E_8$, $F_4$ or $G_2$.
			\item For the other types, the centre $Z(U)$ is isomorphic to a quotient $\mathfrak{R} = S/I$, where $S$ is a polynomial algebra over $\mathbb{K}$ of rank  $|\Psi_{\min}|$, and $I$ is the relation ideal whose generators are specified below
		\end{enumerate}
		\begin{center}
			\resizebox{1\hsize}{!}{$
				\begin{tabular}{|c|l|l|l|}
					\hline
					\textbf{Type} &
					\textbf{Generators $x_{\mu_i}$ of $\operatorname{Im}{\theta}$ } &
					\textbf{Generators of $S$} &
					\textbf{Generating set of $I$}\\
					\hline
					$A_{n \geqslant 2}$ &
					$\begin{aligned}
						& \textsc{z}_1^{r_1},\; \textsc{z}_2^{r_2},\;\cdots,\;\textsc{z}_{n}^{r_n},\\
						& \textsc{z}_1^{d_2} \textsc{z}_2,\;\cdots,\; \textsc{z}_1^{d_n} \textsc{z}_n,\\
						& x_\tau= \textstyle\prod_{i=1}^{n} \textsc{z}_i^{\tau_i} \; (\tau \in T)
					\end{aligned}$
					&
					$\begin{aligned}
						& t_1,\; t_2,\; \cdots,\; t_n,\\
						& p_2,\; \cdots,\; p_n,\\
						& w_{\tau} \; (\tau \in T)
					\end{aligned}$
					&
					$\begin{aligned}
						& t_1^{\frac{d_k}{(n+1,k)}} t_k - p_k^{r_k} \; (k \neq 1), \\
						& \textstyle\prod_{k=2}^n p_k^{\tau_k} - t_1^{\frac{\|\tau\|}{n+1}} w_{\tau}
					\end{aligned}$
					\\
					\hline
					$B_{n}$ &
					$\begin{aligned}
						& \textsc{z}_1,\; \textsc{z}_2,\; \textsc{z}_3,\;\cdots,\; \textsc{z}_{n-1},\;\textsc{z}_{n}^2 \\
					\end{aligned}$
					&
					$\begin{aligned}
						& t_1, t_2, t_3,\cdots,t_{n-1},t_{n}
					\end{aligned}$
					&
					$0$
					\\
					\hline
					$C_{2k+1}$ &
					$\begin{aligned}
						& \textsc{z}_1^2,\; \textsc{z}_2,\; \textsc{z}_3^2,\;\cdots,\; \textsc{z}_{2k},\;\textsc{z}_{2k+1}^2,\\
						& \textsc{z}_{u} \textsc{z}_{v} \; (u,v \; \textnormal{odd},\, u<v)
					\end{aligned}$
					&
					$\begin{aligned}
						& t_1, t_2, t_3,\cdots,t_{2k},t_{2k+1},\\
						& t_{u,v} \; (u,v \; \textnormal{odd},\, u<v)
					\end{aligned}$
					&
					$\begin{aligned}
						& t_1 t_u - t_{1,u}^2 \;(u>1),\\
						& t_1 t_u t_v - t_{1,u} t_{1,v} t_{u,v}\; (u>1)
					\end{aligned}$
					\\
					\hline
					$C_{2k+2}$ &
					$\begin{aligned}
						& \textsc{z}_1^2,\; \textsc{z}_2,\;\cdots,\; \textsc{z}_{2k},\;\textsc{z}_{2k+1}^2,\;\textsc{z}_{2k+2} \\
						& \textsc{z}_{u} \textsc{z}_{v} \; (u,v \; \textnormal{odd},\, u<v)
					\end{aligned}$
					&
					$\begin{aligned}
						& t_1, t_2,\cdots,t_{2k},t_{2k+1}, t_{2k+2}\\
						& t_{u,v} \; (u,v \; \textnormal{odd},\, u<v)
					\end{aligned}$
					&
					$\begin{aligned}
						& t_1 t_u - t_{1,u}^2 \;(u>1),\\
						&  t_1 t_u t_v - t_{1,u} t_{1,v} t_{u,v}\; (u>1)
					\end{aligned}$
					\\
					\hline
					$D_{2k+2}$ &
					$\begin{aligned}
						& \textsc{z}_1^2,\; \textsc{z}_2,\;\cdots,\; \textsc{z}_{2k-2},\;\textsc{z}_{2k-1}^2,\;\textsc{z}_{2k},\; \\
						& \textsc{z}_{2k+1}^2,\; \textsc{z}_{2k+2}^2 \\
						& \textsc{z}_{u} \textsc{z}_{v} \; (u,v \; \textnormal{odd},\; u<v\leqslant2k ),\\
						& \textsc{z}_{2k+1} \textsc{z}_{2k+2}
					\end{aligned}$
					&
					$\begin{aligned}
						& t_1, t_2,\cdots,t_{2k-2},t_{2k-1},t_{2k},\\
						& t_{2k+1}, t_{2k+2},\\
						& t_{u,v} \; (u,v \; \textnormal{odd},\; u<v\leqslant2k ),\\
						& t_{2k+1,2k+2}
					\end{aligned}$
					&
					$\begin{aligned}
						& t_1 t_u - t_{1,u}^2 \;(u>1),\\
						& t_1 t_u t_v - t_{1,u} t_{1,v} t_{u,v}\; (u>1),\\
						& t_{2k+1} t_{2k+2} - t_{2k+1,2k+2}^2
					\end{aligned}$
					\\
					\hline
					$D_{2k+3}$ &
					$\begin{aligned}
						& \textsc{z}_1^2,\; \textsc{z}_2,\;\textsc{z}_3^2,\;\cdots,\; \textsc{z}_{2k},\; \textsc{z}_{2k+1}^2,\;\\
						& \textsc{z}_{2k+2}^4,\; \textsc{z}_{2k+3}^4,\\
						& \textsc{z}_{u} \textsc{z}_{v} \; (u,v \; \textnormal{odd},\; u<v\leqslant 2k+1),\\
						& \textsc{z}_{u} \textsc{z}_{2k+2}^2 \; (u \; \textnormal{odd},\; u\leqslant 2k+1)\\
						& \textsc{z}_{u} \textsc{z}_{2k+3}^2 \; (u \; \textnormal{odd},\; u\leqslant 2k+1)\\
						& \textsc{z}_{2k+2} \textsc{z}_{2k+3}^3,\;\\
						& \textsc{z}_{2k+2}^3 \textsc{z}_{2k+3},\;\\
						& \textsc{z}_{2k+2}^2 \textsc{z}_{2k+3}^2
					\end{aligned}$
					&
					$\begin{aligned}
						& t_1, t_2,t_3,\cdots,t_{2k},t_{2k+1},\\
						& t_{2k+2}, t_{2k+3},\\
						& t_{u,v} \; (u,v \; \textnormal{odd},\; u<v\leqslant 2k+1),\\
						& p_{u} \; (u \; \textnormal{odd},\; u\leqslant 2k+1),\\
						& q_{u} \; (u \; \textnormal{odd},\; u\leqslant 2k+1),\\
						& w_1,\\
						& w_2,\\
						& w_3
					\end{aligned}$
					&
					$\begin{aligned}
						& t_1 t_u - t_{1,u}^2 \;(u>1),\\
						& t_1 t_u t_v - t_{1,u} t_{1,v} t_{u,v}\; (u>1),\\
						& t_1 t_{2k+2} - p_1^2,\; t_1 t_{2k+3} - q_1^2,\\
						& t_{2k+2} t_{1,u} - p_1 p_u \;(u>1),\; \\
						& t_{2k+3} t_{1,u} - q_1 q_u \;(u>1),\\
						& t_{2k+2} t_{2k+3}^3 - w_1^4,\;\\
						& t_{2k+2}^3 t_{2k+3} - w_2^4,\\
						& t_{2k+2} t_{2k+3} - w_3^2
					\end{aligned}$
					\\
					\hline
					$E_6$ &
					$\begin{aligned}
						& \textsc{z}_1^3,\; \textsc{z}_2,\;  \textsc{z}_3^3, \;
						\textsc{z}_4,\; \textsc{z}_5^3, \; \textsc{z}_6^3, \;\\
						& \textsc{z}_1 \textsc{z}_3,\; \textsc{z}_1 \textsc{z}_6,\;
						\textsc{z}_3 \textsc{z}_5,\; \textsc{z}_5 \textsc{z}_6\\
						& \textsc{z}_1 \textsc{z}_5^2,\; \textsc{z}_1^2 \textsc{z}_5,\;
						\textsc{z}_3 \textsc{z}_6^2,\; \textsc{z}_3^2 \textsc{z}_6
					\end{aligned}$
					&
					$\begin{aligned}
						&t_1, t_2, t_3, t_4, t_5, t_6, \\
						&t_7, t_8, t_9, t_{10}, \\
						&t_{11}, t_{12}, t_{13}, t_{14}
					\end{aligned}$
					&
					$\begin{aligned}
						& t_1 t_3 - t_7^3,\;
						t_1 t_6 - t_8^3,\;\\
						& t_3 t_5 - t_9^3,\;
						t_8 t_9 - t_7 t_{10}, \\
						& t_7 t_9^2 - t_3 t_{11},\;
						t_7^2 t_9 - t_3 t_{12}, \;\\
						& t_7 t_8^2-t_1 t_{13}, \;
						t_7^2t_8-t_1t_{14}
					\end{aligned}$
					\\
					\hline
					$E_7$ &
					$\begin{aligned}
						& \textsc{z}_1,\; \textsc{z}_2^2,\;  \textsc{z}_3, \;
						\textsc{z}_4,\; \textsc{z}_5^2, \; \textsc{z}_6,\; \textsc{z}_7^2, \\
						& \textsc{z}_2 \textsc{z}_5,\; \textsc{z}_2 \textsc{z}_7,\;
						\textsc{z}_5 \textsc{z}_7
					\end{aligned}$
					&
					$\begin{aligned}
						& t_1, t_2, t_3, t_4, t_5, t_6, t_7,\\
						& t_8, t_9, t_{10}
					\end{aligned}$
					&
					$\begin{aligned}
						& t_2 t_5 - t_8^2, \quad
						t_2 t_7 - t_9^2, \\
						& t_2 t_5 t_7 - t_8 t_9 t_{10}
					\end{aligned}$
					\\
					\hline
					$E_8\;F_4\;G_2$ &
					$\begin{aligned}
						& \textsc{z}_1,\; \textsc{z}_2,\;\cdots,\;\textsc{z}_{n}
					\end{aligned}$
					&
					$\begin{aligned}
						& t_1, t_2,\cdots,t_{n}
					\end{aligned}$
					&
					$0$
					\\
					\hline
				\end{tabular}
				$}
		\end{center}
	\end{theorem}
	
	\begin{proof}
		(1) When $\mathfrak{g}$ is of type $B_n$, $E_8$, $F_4$, or $G_2$, Lemmas \ref{alg ind} and \ref{minigenset} yield
		\begin{equation*}
			Z(U) \cong \operatorname{Im}(\theta) =
			\begin{cases}	
				\langle \textsc{z}_{1} ,\cdots, \textsc{z}_{n-1}, \textsc{z}_{n} \rangle, \quad \text{for types $E_8,F_4$, and $G_2$,} \\
				\langle \textsc{z}_{1} ,\cdots, \textsc{z}_{n-1}, \textsc{z}^2_{n} \rangle, \quad \text{for type $B_n$,}
			\end{cases}
		\end{equation*}
		and $\{\textsc{z}_{i}\}_{i=1}^n$ is an algebraically independent set.
		Therefore, the centre $Z(U)\cong\operatorname{Im}(\theta)$ is isomorphic to a polynomial algebra of rank $n$.

		(2.1) From type $C_n$ to $E_7$, we only show the proof of type $C_{2k+1}$, as the other types follow analogously.
		
		Define an algebra epimorphism $\phi: S \rightarrow \operatorname{Im}(\theta)$ by
		\begin{equation*}
			\phi(t_{2i-1}) = \textsc{z}_{2i-1}^2,\ \phi(t_{2i}) = \textsc{z}_{2i}, \ (1 \le i \le k+1),\
			\phi(t_{u,v}) = \textsc{z}_u \textsc{z}_v, \ \textnormal{for odd } u < v \leqslant 2k+1,
		\end{equation*}
		and one can directly check that $ I \subseteq \ker \phi $.
		
		Let $S_1 = \mathbb{K}[t_1, \dots, t_{2k+1}]$ and $S_2 = \mathbb{K}[\textsc{z}_1^2, \textsc{z}_2, \dots, \textsc{z}_{2k}, \textsc{z}_{2k+1}^2]$. Since $\textsc{z}_1^2, \textsc{z}_2, \dots, \textsc{z}_{2k}$, $\textsc{z}_{2k+1}^2$ are algebraically independent, the map $\phi|_{S_1}$ is an algebra isomorphism from $S_1$ to $S_2$, and $S_1 \cap \ker \phi = \{0\}$.
		Now let $T_i := S_i \setminus \{0\}$ for $i = 1, 2$. It is obvious that $\phi$ can be extended to an
		algebra epimorphism $\overline{\phi}:T_1^{-1}S\rightarrow T_2^{-1}Im(\theta)$ as follows
		\begin{equation*}
			\overline{\phi}(a^{-1}b) = \phi(a)^{-1}\phi(b), \quad \forall a \in T_1, \, b \in S.
		\end{equation*}
		There is a canonical embedding $T_1^{-1}: S \hookrightarrow T_1^{-1}S$ via $s\mapsto \frac{s}{1}$, and we define $J := T_1^{-1}I \subseteq \ker \overline{\phi}$.
		
		We next prove $ I \supseteq \ker \phi $ by showing that the quotient ring $T^{-1}_1S/J$ is a field.
		
		Let $ F_0 = \mathbb{K}(t_1,\cdots,t_{2k+1}) = T_1^{-1} S_1 $, the function field of $ S_1 $. Then $ F_1 := F_0[t_{2k+2}] \cong F_0[t]/(t^2 - t_1 t_3) $.
		$ F_1 $ is a field since $ t^2 - t_1t_3 $ is irreducible in ring $ F_0[t] $.
		Similarly, we have $ F_2 = F_1[t_{1,5}] = F_1[t]/(t^2-t_1t_5), \cdots, F_k = F_{k-1}[t_{1,2k+1}]= F_{k-1}[t]/(t^2-t_1t_{2k+1}) $
		and $ F_2, \ldots, F_k $ all are field.
		Moreover, by $t_{u,v}=t_{1,u}^{-1}t_{1,v}^{-1}t_1t_ut_v$, we have all $t_{u,v}\in  F_k$, that is, $T_1^{-1}S/J=F_k$. Now $T_1^{-1}S/J$ is a field, it implies that $J$ is a maximal ideal, then (i) $\ker(\overline{\phi})$ has to be the maximal ideal $J$, and $J$ a prime ideal in ${T_1}^{-1} S$, which implies (ii) $I$ is prime in $S$ since it is the contraction of a prime ideal $J$.
		Finally, notice that $I\cap T_1= \emptyset$, we have $J\cap S= T_1^{-1} I \cap S \overset{\text{(ii)}}{=} I$.
		Then one can obtain that
		\begin{equation*}
			\ker \phi \subseteq \ker \overline{\phi} \cap S \overset{\text{(i)}}{=} J \cap S \overset{\text{(ii)}}{=} I.
		\end{equation*}
		Above all, we have $\ker \phi = I$, and then $\mathfrak{R} = S/I \cong \operatorname{Im}(\theta) \cong Z(U)$.

		(2.2) For type $A_n$, define an algebra epimorphism $\phi : S \to \operatorname{Im}(\theta)$ via
		\begin{equation*}
			\phi(t_i) = \textsc{z}_i^{r_i}, \quad \phi(p_k) = \textsc{z}_1^{d_k} \textsc{z}_k, \quad \phi(w_{\tau}) = x_\tau= \textstyle\prod_{i=1}^n{\textsc{z}_{i}^{\tau_i}}
		\end{equation*}
		for $1 \le i, k \le n \; (k \ne 1)$ and all $\tau = \sum_{i=1}^{n} \tau_i \varpi_i \in T$.
		Then by Lemma \ref{norm of tau}, we have $\frac{\left\| \tau \right\|}{n+1} \in \mathbb{Z}_{>0}$ and
		\begin{equation*}
			\begin{split}
				&\phi \left( p_{k}^{r_k} \right) =\textsc{z}_{1}^{r_kd_k} \textsc{z}_{k}^{r_k}=\textsc{z}_{1}^{\frac{\left( n+1 \right)d_k}{\left( n+1,k \right)}}\textsc{z}_{k}^{r_k}=\phi \left( t_{1}^{\frac{d_k}{\left( n+1,k \right)}} t_k \right) ,\\
				&\phi \left(\textstyle\prod_{k\ne 1}{p_{k}^{\tau_k}}\right)=\textstyle\prod_{k\ne 1}{\textsc{z}_{1}^{d_k \tau_k} \textsc{z}_{k}^{\tau_k}}=\textsc{z}_{1}^{\left\| \tau \right\|}\textstyle\prod_{k=1}^n{\textsc{z}_{k}^{\tau_k}}=\phi \left( {t_1}^{\frac{\left\| \tau \right\|}{n+1}}w_{\tau} \right),
			\end{split}
		\end{equation*}
		that is, $ I \subseteq \ker \phi$.

		In order to prove $\phi$ is injective, i.e., $ I \supseteq \ker \phi$, similarly we let $S_1=\mathbb{K}[t_1,\cdots,t_n]$, $S_2=\mathbb{K}[\textsc{z}_1^{r_1},\cdots,\textsc{z}_n^{r_n}]$, $T_i := S_i \setminus \{0\}$ for $i = 1, 2$, and repeat the process in (2.1). It suffices to prove that $T_1^{-1} S/J$ is a field, where $J := T_1^{-1}I \subseteq \ker \overline{\phi}$.
		
		Let $F_0 = \mathbb{K}(t_1, \cdots, t_n)$ be the fraction field of $S_1$. Then $F_1 := F_0[p_2]$ is a field since $t^{r_2} - t_1^{\frac{d_2}{(n+1,2)}} t_2$ is irreducible in $F_0[t]$, and $F_2 := F_1[p_2] = F_0[p_2, p_3]$ is also a field since $t^{r_3} - t_1^{\frac{d_3}{(n+1,3)}} t_3$ is irreducible in $F_1[t]$. In the same way, $F_{n-1} := F_0[p_2, \cdots, p_n]$ is a field.
		Meanwhile, we have $w_{\tau} = t_1^{-\frac{\|\tau\|}{n+1}} \prod_{k=2}^n p_k^{\tau_k} \in F_{n-1}$ for all $\tau \in T$. Thus, $T_1^{-1} S/J = F_{n-1}$, which is indeed a field.
		
		This completes the proof of $S/I \cong \operatorname{Im}(\theta)$.
	\end{proof}	
	
	\begin{theorem}
		The result above provides the same descriptions of $\Psi=\Lambda^+\cap Q$ and $(U^0_\flat)^W$ for the two-parameter quantum group $U_{r,s}(\mathfrak{g})$ in \cite{HW25}. That is,
		\begin{equation*}
			\xi(Z(U_{r,s}))
			\begin{cases}	
				= (U^0_\flat)^W   \hspace{1.9cm} \cong  S/I, &\textnormal{if $n$ is even,}\\
				\supseteq (U^0_\flat)^W \otimes \mathbb{K}[z_{*},z_{*}^{-1}] \cong S/I \otimes \mathbb{K}[z_{*},z_{*}^{-1}],
				&\textnormal{if $n$ is odd,}
			\end{cases}
		\end{equation*}
		where the polynomial algebra $S$ and the ideal $I$ is shown in the table of Theorem \ref{Last theorem}.
	\end{theorem}

	\begin{example} We provide an explicit description of $Z(U_{q,G})$ for $\mathfrak{g}$ of types $A_2$, $A_3$, and $A_4$.
		\begin{center}
			\begin{tabular}{|c|l|}
				\hline
				\textbf{Type} &
				\textbf{Description of $Z(U_{q,G})$}\\
				\hline
				$A_2$ &	
				$\begin{aligned}
					& \Psi_{\min} = \{\, \text{single}:  3\varpi_1,\, 3\varpi_2,\, \text{special}: \varpi_1 + \varpi_2 \,\} \\
					& Z(U_{q,G})\cong \mathbb{K}[t_1, t_2, p_2]/(t_1 t_2 - p_2^3)
				\end{aligned}$
				\\
				\hline
				$A_3$ &
				$\begin{aligned}
					& \Psi_{\min} = \{\,\text{single}: 4\varpi_1,\, 2\varpi_2,\, 4\varpi_3,\,\text{special}: 2\varpi_1 + \varpi_2,\, \varpi_1 + \varpi_3,\,T: \varpi_2 + 2\varpi_3 \,\} \\
					& Z(U_{q,G})\cong \mathbb{K}[t_1, t_2, t_3, p_2, p_3, w_1]\, /\, (p_2^2 - t_1 t_2,\ p_3^4 - t_1 t_3,\ t_1 w_1 - p_2 p_3^2)
				\end{aligned}$
				\\
				\hline
				$A_4$ &
				$\begin{aligned}
					& \Psi_{\min} = \{\, \text{single}: 5\varpi_1,\, 5\varpi_2,\, 5\varpi_3,\, 5\varpi_4,\,  \\
					& \quad\quad\quad\quad \text{special}: 3\varpi_1+\varpi_2,\, 2\varpi_1+\varpi_3,\, \varpi_1+\varpi_4, \\
					& \quad\quad\quad\quad T: \varpi_2+\varpi_3,\, \varpi_1+2\varpi_2,\, \varpi_1+3\varpi_3,\, \varpi_2+2\varpi_4,\,\\
					& \quad\quad\quad\quad\quad\ 3\varpi_2+\varpi_4,\, 2\varpi_3+\varpi_4,\, \varpi_3+3\varpi_4 \}\\
					& Z(U_{q,G})\cong \mathbb{K}[t_1, t_2, t_3, t_4,\ p_2, p_3, p_4,\ w_1, w_2, w_3, w_4, w_5, w_6, w_7]/I,\\
					& I = \langle\; t_1^3 t_2 - p_2^5,\quad t_1^2 t_3 - p_3^5,\quad t_1 t_4 - p_4^5,\quad t_1 w_1 - p_2 p_3,\quad t_1 w_2 - y_2^2, \\
					& \quad\quad t_1 w_3 - p_3^3,\quad t_1 w_4 - p_2 p_4^2,\quad t_1 w_5 - p_2^3 p_4,\quad t_1 w_6 - p_3^2 p_4,\quad t_1 w_7 - p_3 p_4^3 \;\rangle
				\end{aligned}$
				\\
				\hline
			\end{tabular}
		\end{center}
		We compare these results with those in \cite{LXZ16} and write $\Psi'= \Lambda^+ \cap \frac{Q}{2} $. When $n$ is even, we have $Z(U_{q,G}(\mathfrak{sl}_{n+1}))\cong Z(U_{q}(\mathfrak{sl}_{n+1}))$ since monoid $\Psi = \Psi'$. When $n$ is odd, we prove that $\Psi\ncong\Psi'$ as follows.
		
		Assume there exists \, $\varphi: \Psi \cong \Psi'$ as monoids, then $\varphi$ bijectively maps their minimal generating systems. Let $r=\frac{n+1}{2}$, notice that  $\varpi_r\in \Psi_{\min}^\prime$, then $\exists\; \alpha \in \Psi_{\min},\;\varphi(\alpha)=\varpi_r$. This forces that $\varphi(\beta)=\sum_{i\ne r}{k_i\varpi_i}$ for all $\beta\in \Psi_{\min}$, $\beta \ne \alpha$ (otherwise, $\varphi(\beta)$ is decomposable). Therefore, $\varpi_r$ cannot appear as an additive component in any $\mathbb{Z}$-multiple of $\varphi(\beta)$.

		If $\alpha = \textbf{s}(1) =(n+1)\varpi_1$, then
		\begin{equation*}
			r_2 \varphi(\textbf{e}(2)) =
			r_2 \varphi(d_2\varpi_1+\varpi_2)
			=\frac{d_2}{(n{+}1,2)}\varphi((n+1)\varpi_1) + \varphi(r_2\varpi_2)
			=\frac{d_2}{(n{+}1,2)}\varpi_r + \varphi(\textbf{s}(2)).
		\end{equation*}
		If $\alpha= \textbf{s}(k) = r_k\varpi_k (k > 1)$, then
		\begin{equation*}
			r_k \varphi(\textbf{e}(k)) =
			r_k \varphi(d_k\varpi_1+\varpi_k)=
			\frac{d_k}{(n{+}1,k)}\varphi((n+1)\varpi_1)+\varphi(r_k \varpi_k)=
			\frac{d_k}{(n{+}1,k)}\varphi(\textbf{s}(1))+\varpi_r.
		\end{equation*}
		If $\alpha=\textbf{e}(k)= d_k\varpi_1+\varpi_k (k > 1)$, then
		\begin{equation*}
			r_k\varpi_r=
			r_k\varphi(\textbf{e}(k))=
			\frac{d_k}{(n+1,k)}\varphi(\textbf{s}(1))+\varphi(\textbf{s}(k)).
		\end{equation*}
		If $\alpha = \sum_{i=1}^{n} \alpha_i \varpi_i\in T$, applying $\varphi$ to Lemma \ref{norm of tau} yields
		\begin{equation*}
			\textstyle\sum_{k \neq 1} \alpha_k \varphi(\mathbf{e}(k)) =
			\varphi(\alpha) + \frac{\|\alpha\|}{n+1}\varphi((n+1)\varpi_1) = \varpi_r + \frac{\|\alpha\|}{n+1} \varphi(\textbf{s}(1)).
		\end{equation*}
		In all cases, contradictions arise in which $\varpi_r$ appears in some $\mathbb{Z}$-multiple of $\varphi(\beta)$ for $\beta \in \Psi_{\min}$ with $\beta \ne \alpha$.
		It turns out that $\Psi \ncong \Psi'$ and hence $Z(U_{q,G}(\mathfrak{sl}_{n+1})) \ncong Z(U_{q}(\mathfrak{sl}_{n+1}))$ when $n>1$ is odd.
	\end{example}
	
	\section*{Appendix A}
	
	The fundamental weights $\varpi_i$ ($i = 1, \dots, n$) have the following $\mathbb{Q}$-linear expressions in terms of the simple roots $\{\alpha_i\}_{i=1}^n$:
	\begin{center}
		\resizebox{1\hsize}{!}{$
			\begin{tabular}{|c|c|c|c|}
				\hline
				$A_n$ & \multicolumn{3}{c|}{
					$\varpi_i = \left(\frac{(n-i+1)}{n+1}, \frac{2(n-i+1)}{n+1}, \cdots,\frac{(i-1)(n-i+1)}{n+1}, \frac{i(n-i+1)}{n+1}, \frac{i(n-i)}{n+1}, \cdots, \frac{i}{n+1}\right)$
				}\\
				\hline
				\multirow{2}{*}{$B_n$} & $\varpi_i = (1, 2, \cdots, i-1, i, i, \cdots, i)$ & \multirow{3}{*}{$D_n$} & $\varpi_i = \left(1,2,\cdots,i-1,i,i,\cdots,i,\frac{i}{2},\frac{i}{2}\right)$\\
				& $\varpi_n = \frac{1}{2}(1, 2, \cdots, n)$ & & $\varpi_{n-1} = \frac{1}{2} \left(1,2,\cdots,n-2,\frac{n}{2},\frac{n-2}{2}\right)$\\
				\cline{1-2}
				$C_n$ & $\varpi_i = \left(1,\; 2,\; \cdots,\; i-1,\; i, \; i \;\cdots\; i, \; \frac{i}{2} \right)$ & & \hf\hf $\varpi_n = \frac{1}{2} \left(1,2,\cdots,n-2,\frac{n-2}{2},\frac{n}{2}\right)$ \\	
				\hline
				\multirow{6}{*}{$E_6$} & $\varpi_1 = \frac{1}{3}( 4, \hf 3, \hf 5, \hf 6, \hf 4, \hf 2)$ & \multirow{6}{*}{$E_7$} & $\varpi_1 = \hf ( 2, \hf 2, \hf 3, \hf 4, \hf 3, \hf 2, \hf 1)$ \\
				& $\varpi_2 = \hf ( 1, \hf 2, \hf 2, \hf 3, \hf 2, \hf 1)$ & & $\varpi_2 = \frac{1}{3}( 4, \hf 7, \hf 8, 12, \hf 9, \hf 6, \hf 3)$ \\
				& $\varpi_3 = \frac{1}{3}( 5, \hf 6, 10, 12, \hf 8, \hf 4)$ & & $\varpi_3 = \hf ( 3, \hf 4, \hf 6, \hf 8, \hf 6, \hf 4, \hf 2) $\\
				& $\varpi_4 = \hf ( 2, \hf 3, \hf 4, \hf 6, \hf 4, \hf 2)$ & & $\varpi_4 = \hf ( 4, \hf 6, \hf 8, 12, \hf 9, \hf 6, \hf 3)$ \\
				& $\varpi_5 = \frac{1}{3}( 4, \hf 6, \hf 8, 12, 10, \hf 5)$ & & $\varpi_5 = \frac{1}{3}( 6, \hf 9, 12, 18, 15, 10, \hf 5)$ \\
				& $\varpi_6 = \frac{1}{3}( 2, \hf 3, \hf 4, \hf 6, \hf 5, \hf 4)$ & & $\varpi_6 = \hf ( 2, \hf 3, \hf 4, \hf 6, \hf 5, \hf 4, \hf 2)$\\
				& & & $\varpi_7 = \frac{1}{2}( 2, \hf 3, \hf 4, \hf 6, \hf 5, \hf 4, \hf 3)$ \\
				\hline
				\multirow{8}{*}{$E_8$} & $\varpi_1 = (\hf 4, \hf 5, \hf 7, 10, \hf 8, \hf 6, \hf 4, 2)$ & \multirow{6}{*}{$F_4$} & \\
				& $\varpi_2 = (\hf 5, \hf 8, 10, 15, 12, \hf 9, \hf 6, 3)$ & & $\varpi_1 = (2, 3, 4, 2)$ \\
				& $\varpi_3 = (\hf 7, 10, 14, 20, 16, 12, \hf 8, 4)$ & & $\varpi_2 = (3, 6, 8, 4)$ \\
				& $\varpi_4 = (10, 15, 20, 30, 24, 18, 12, 6)$ & & $\varpi_3 = (2, 4, 6, 3)$ \\
				& $\varpi_5 = (\hf 8, 12, 16, 24, 20, 15, 10, 5)$ & & $\varpi_4 = (1, 2, 3, 2)$ \\
				& $\varpi_6 = (\hf 6, \hf 9, 12, 18, 15, 12, \hf 8, 4)$ &  &  \\
				\cline{3-4}
				& $\varpi_7 = (\hf 4, \hf 6, \hf 8, 12, 10, \hf 8, \hf 6, 3)$ & \multirow{2}{*}{$G_2$}&$\varpi_1 = (2, 1)$ \\
				& $\varpi_8 = (\hf 2, \hf 3, \hf 4, \hf 6, \hf 5, \hf 4, \hf 3, 2)$ & &$\varpi_2 = (3, 2) $  \\
				\hline
			\end{tabular}
			$}
	\end{center}

\end{document}